\documentclass[reqno]{amsart}
\usepackage{amssymb}
\usepackage{amsmath}
\usepackage[mathscr]{euscript}
\usepackage{times}
\usepackage{graphicx}
\usepackage{color}
\usepackage[all]{xy}

\makeatletter
\@addtoreset{equation}{section}
\makeatother

\newtheorem{theorem}{Theorem}[section]
\newtheorem{lemma}[theorem]{Lemma}
\newtheorem{proposition}[theorem]{Proposition}
\newtheorem{corollary}[theorem]{Corollary}

\newtheorem{remark}[theorem]{Remark}
\newtheorem{example}[theorem]{Example}

\newcounter{as}[section]

\renewcommand{\a}{\alpha}
\renewcommand{\b}{\beta}

\newcommand{\ga}{\gamma}

\newcommand{\si}{\sigma}

\newcommand{\om}{\omega}

\newcommand{\De}{\Delta}

\newcommand{\La}{\Lambda}
\newcommand{\Om}{\Omega}
\newcommand{\lan}{\langle}
\newcommand{\ran}{\rangle}

\renewcommand{\tilde}{\widetilde}

\def\R{{\mathbb{R}}}

\def\X{{\mathbb{X}}}

\def\N{{\mathbb{N}}}
\def\Z{{\mathbb{Z}}}

\def\B{{\mathcal{B}}}

\def\L{{\mathcal{L}}}
\def\K{{\mathcal{K}}}
\def\F{{\mathcal{F}}}
\def\rank{\mathrm{rank\ }}
\def\tm{{\lfloor t \rfloor_m}}

\begin{document}

\title{Limit theorems for random cubical homology}
\author{Yasuaki Hiraoka and Kenkichi Tsunoda}

\address{\noindent Advanced Institute for Materials Research, Tohoku University, 2-1-1, Katahira, Aoba-ku, Sendai, 980-8577, Japan. 
\newline Center for Advanced Intelligence Project, RIKEN, Tokyo, 103-0027, Japan.
\newline e-mail: \rm \texttt{hiraoka@tohoku.ac.jp} }

\address{\noindent Department of Mathematics, Graduate School of Science,
Osaka University, 1-1, Machikaneyama-cho, Toyonaka, Osaka, 560-0043, Japan.
\newline Center for Advanced Intelligence Project, RIKEN, Tokyo, 103-0027, Japan.
\newline e-mail:  \rm \texttt{k-tsunoda@math.sci.osaka-u.ac.jp}}

\subjclass[2010]{{60D05, 52C99, 60F05, 60F15}}

\keywords{Random topology, Cubical complex, Cubical homology, Betti number}

\begin{abstract}
This paper studies random cubical sets in $\R^d$. Given a cubical set $X \subset \R^d$, a random variable $\omega_Q\in[0,1]$ is assigned for each elementary cube $Q$ in $X$, and a random cubical set $X(t)$ is defined by the sublevel set of $X$ consisting of elementary cubes with $\omega_Q\leq t$ for each  $t\in[0,1]$. Under this setting,  the main results of this paper show the limit theorems (law of large numbers and central limit theorem) for Betti numbers and lifetime sums of random cubical sets and filtrations. In addition to the limit theorems, the positivity of the limiting Betti numbers is also shown. 
\end{abstract}

\maketitle

\section{Introduction}\label{sec:intro}
The mathematical subject studied in this paper is motivated by imaging science. Objects in $\R^d$ are usually represented by {\em cubical sets}, which are the union of pixels or higher dimensional voxels called elementary cubes, and those digitalized images are used as the input for image processing. For example, image recognition techniques identify patterns and characteristic shape features embedded in those digital images. Recent progress on computational topology \cite{eh,kmm} 
allows us to utilize {\em homology} as a descriptor of the images. Here, homology is an algebraic tool to study holes in geometric objects, and it enables us to extract global topological features in data (e.g., \cite{akkmop,kimura,kms}).  In particular, the mathematical framework mentioned above is called {\em cubical homology}, which will be briefly explained in Subsection \ref{ch} (see \cite{kmm} for details).

In applications, digital images usually contain measurement/quantization noise, and hence it is important to estimate the effect of randomness on cubical homology. Furthermore, by studying the asymptotic behaviors of randomized cubical homology, we can understand the (homological) structures in the images as a difference from the random states. {\em Random topology} is a mathematical subject to study these problems, and it is a new branch of mathematics that has emerged in the intersection between algebraic topology and probability theory. The reader may refer to the survey papers \cite{bk,k} and the references therein for further details. 

In this paper, we study several limit theorems for random cubical homology. Our model, which will be precisely explained in Subsection \ref{rcs},  assigns a random variable $\omega_Q$ on $[0,1]$ from a probability measure $P$ for each elementary cube $Q$ in a cubical set $X\subset \R^d$. Then, for each $t$, we construct a random cubical set as a sublevel set $X(t)=\{Q\in X\colon\omega_Q\leq t\}$. This is a natural higher dimensional generalization of the classical bond percolation model \cite{g} in which the elementary cubes only consist of vertices and edges in $\Z^d$. This model also contains the Bernoulli random cubical complex \cite{hs2}, which is an analog of the Linial-Meshulam random simplicial complex \cite{lm}. Furthermore, the sequence $\{X(t)\}_{0\leq t\leq 1}$ can be regarded as a random filtration\footnote{In this paper, the term \lq\lq filtration\rq\rq\  means an increasing family of cubical complexes as usual in topology.} and this naturally connects to the concept of {\em persistent homology} \cite{elz,zc}. 

Under this setting, we study the law of large numbers (LLN) and the central limit theorem (CLT) for the Betti numbers $\beta_q(t):=\beta_q(X(t))$. We also study the asymptotic behavior of the lifetime sum of persistent homology. Here, given a filtration such as $\{X(t)\}_{0\leq t\leq 1}$, the $q$-th persistent homology characterizes each $q$-dimensional hole $c$ by encoding its birth and death parameters $t=b,d$ ($0\leq b\leq d\leq 1$), where $b$ and $d$ express the appearance and disappearance of the hole $c$, respectively. Then, the lifetime sum $L_q$ is defined by the sum of the lifetime $d-b$ for all $q$-dimensional holes. Alternatively, it can be also expressed by the integral $L_q=\int_0^1\beta_q(t)dt$ of the Betti number. The lifetime sum has a natural connection to the classical theorem in combinatorial probability theory called Frieze's $\zeta(3)$ theorem \cite{f}, and a higher dimensional generalization of Frieze's theorem is studied in \cite{hk,hs1,hs2}. In view of this connection, we also prove the LLN and the CLT for the lifetime sum of the persistent homology in this paper. 

We remark on several papers related to this work. First, the paper \cite{hs2} studies the Bernoulli random cubical complex and determines the asymptotic order of the expected lifetime sum of the persistent homology. Hence, the LLN for the lifetime sum in this paper is stronger than the result in \cite{hs2}. We also remark that the paper \cite{ww} discusses several random cubical sets, and gives exact polynomial formulae for the expected value and variance of the intrinsic volumes. The study on cubical homology is addressed in the future work in \cite{ww}. Instead of discrete settings, the limit theorems for Betti numbers defined on point processes in $\R^d$ are studied in \cite{ysa}. The paper \cite{dhs} also shows the limit theorems for persistence diagrams on point processes in $\R^d$.

The paper is organized as follows. In Section \ref{model_result}, after brief introduction of cubical sets and cubical homology, we explain our model of random cubical sets and state the main results of the LLN and the CLT for Betti numbers and lifetime sums. Some computations of limiting Betti numbers are also presented at the end of this section. The proofs for the LLN and the CLT are given in Section \ref{sec:LLN} and \ref{sec:CLT}, respectively. At the end of Section \ref{sec:LLN}, we also show a sufficient condition for the positivity of the limiting Betti numbers. Section \ref{sec:conclusions} concludes the paper and shows some future problems.

\section{Model and main results}\label{model_result}

\subsection{Cubical homology}\label{ch}
We review in this subsection the concept of cubical homology,
which is a main object studied in this paper.
This subsection is devoted to a brief summary of Chapter 2 in \cite{kmm}.
We refer to \cite{kmm} for more detailed description and the proofs of propositions introduced in this subsection.

An {\em elementary interval} is a closed interval $I \subset \R$ of the form
\begin{equation*}
I \;=\; [l,l+1] \qquad \text{or} \qquad I \;=\; [l,l] \;,
\end{equation*}
for some $l$ in $\Z$. In the latter case, we shall write it as $I=[l]$.
Fix $d\in\N$. Throughout the paper, $d$ represents the dimension of the state space
where cubical sets, which will be defined later, are considered.
An {\em elementary cube} $Q$ is a finite product of elementary intervals of the form
\begin{equation*}
Q \;=\; I_1 \times \cdots \times I_d \;,
\end{equation*}
where $I_i \subset \R$ is an elementary interval for each $1 \le i \le d$.

Let $Q$ be an elementary cube of the form
$Q = I_1 \times \cdots \times I_d$ with elementary intervals $I_i$, $1\le i \le d$.
For each $1 \le i \le d$, denote by $I_i(Q)$ the $i$-th component of $Q$: $I_i(Q) = I_i$.
For an elementary interval $I\subset\R$, $I$ is said to be nondegenerate if $I$ is not a singleton.
We denote by $\dim Q$ the number of nondegenerate components of $Q$:
\begin{equation*}
\dim Q \;=\; \#\{1\le i \le d : \text{$I_i(Q)$ is nondegenerate} \} \;.
\end{equation*}
Define $\K^d$ by the set of all elementary cubes in $\R^d$.
For each $0 \le k \le d$, we also define $\K_k^d$ as the set of all elementary cubes 
in $\R^d$ whose dimension is equal to $k$.

A subset $X \subset \R^d$ is said to be {\em cubical} if $X$ can be written as a union of
elementary cubes in $\R^d$. Note that a cubical set is a subset of $\R^d$ with this definition,
not a set of chains (see also Remark \ref{rev1}).
Note that an infinite union of elementary cubes in $\R^d$
is also included in our definition of cubical sets although it is assumed to be a finite union in \cite{kmm}.
For a cubical set $X \subset \R^d$, denote by $\K^d(X)$ the set of all elementary cubes contained in $X$.
For each $0 \le k \le d$, we also denote by $\K^d_k(X)$ the set of all elementary cubes in $X$
whose dimension is equal to $k$.

For each elementary cube $Q$ in $\K_k^d$, let $\widehat Q$ be an algebraic object of $Q$.
$\widehat Q$ is called an {\em elementary $k$-chain}.
Denote by $\widehat\K_k^d$ the set of all elementary $k$-chains.
We define the $\Z$-free module $C_k^d = \Z(\widehat\K_k^d)$
by the module over $\Z$ generated by all elementary $k$-chains:
\begin{equation*}
C_k^d \;:=\; \{ c = \sum_{\text{finite sum}} \a_i \widehat Q_i : \widehat Q _i \in \widehat\K_k^d,\  \a_i \in \Z \} \;.
\end{equation*}
An element belonging to  $C_k^d$ is called a {\em $k$-dimensional chain}.
We also set $C_k^d:=0$ for $k<0$ or $k>d$.
For a cubical set $X \subset \R^d$, we similarly denote by $\widehat\K_k^d(X)$
the set of all elementary $k$-chains in $X$
and define the $\Z$-free module $C_k^d(X) = \Z(\widehat\K_k^d(X))$, respectively:
\begin{align*}
\widehat\K_k^d(X) & \;:=\; \{ \widehat Q : Q \in \K_k^d(X) \} \;, \\
C_k^d(X) & \;:=\; \{ c = \sum_{\text{finite sum}} \a_i \widehat Q_i : \widehat Q _i \in \widehat\K_k^d(X),\  \a_i \in \Z \} \;.
\end{align*}

Consider $k$-dimensional chains $c_1, c_2 \in C_k^d$ with $c_1 = \sum_{i=1}^m \a_i \widehat Q_i$
and $c_2 = \sum_{i=1}^m \b_i \widehat Q_i$.
The scalar product of the chains $c_1$ and $c_2$ is defined as
\begin{equation*}
\lan c_1, c_2\ran \;=\; \sum_{i=1}^m \a_i \b_i \;.
\end{equation*}
For elementary cubes $P\in \K_k^d$ and $Q\in\K_{k'}^{d'}$, 
we define the cubical product $\widehat P \diamond \widehat Q$ by
\begin{equation*}
\widehat P \diamond \widehat Q \;:=\; \widehat{P \times Q} \;.
\end{equation*}
Note that, for elementary cubes $P \in \K_k^d$ and $Q \in \K_{k'}^{d'}$, 
the direct product $P \times Q$ is also an elementary cube belonging to $\K_{k+k'}^{d+d'} $.
Therefore the cubical product $\widehat P \diamond \widehat Q$ is well-defined
and defines an element in $\widehat\K_{k+k'}^{d+d'}$.
For general chains $c_1 \in C_k^d$ and $c_2 \in C_{k'}^{d'}$, we define
the cubical product $c_1\diamond c_2\in C_{k+k'}^{d+d'}$ as
\begin{equation*}
c_1\diamond c_2 \;=\; \sum_{P \in \K_k^d}\sum_{Q \in \K_{k'}^{d'}}
\lan c_1,\widehat P \ran \lan c_2, \widehat Q \ran \widehat{P \times Q} \;.
\end{equation*}

Let $k\in\Z$. The {\em cubical boundary operator} $\partial_k:C_k^d\to C_{k-1}^d$,
which is a homomorphism of $\Z$-modules, is defined by the following way.
We first set $\partial_k:=0$ if $C_k^d=0$ or $C_{k-1}^d=0$.
Let $Q$ be an elementary cube in $\K_k^d$.
For $k=1$, the boundary operator is defined as
\begin{align*}
\partial_1\widehat{Q} \;=\;
\begin{cases}
\widehat{[l+1]} - \widehat{[l]} \;, \qquad & \text{if $Q=[l,l+1]$ for some $l\in\Z$ ,} \\
0 \;, & \text{otherwise} \;.
\end{cases}
\end{align*}
We assume that the boundary operator $\partial_k$
for all $k$-dimensional elementary chains is already defined for some $k\ge1$ and assume that $\dim Q= k+1$.
Let $I=I_1(Q)$ and $P=I_2(Q)\times\cdots\times I_d(Q)$.
We then define $\partial_{k+1}\widehat Q$ by
\begin{equation*}
\partial_{k+1}\widehat Q \;:=\; (\partial_{k_1} \widehat I) \diamond \widehat P
+ (-1)^{\dim I} \widehat I \diamond (\partial_{k_2} \widehat P) \;,
\end{equation*}
where $k_1= \dim I$ and $k_2= \dim P$.
We finally extend the definition to all chains by linearity, that is,
for a chain $c$ in $C_k^d$ with the form $c = \a_1\widehat{Q}_1 + \cdots + \a_m\widehat{Q}_m$, define
\begin{equation*}
\partial_k c \;:=\; \a_1 \partial_k \widehat{Q}_1 + \cdots + \a_m \partial_k \widehat{Q}_m \;.
\end{equation*}
 
We introduce in the following proposition an alternative formula for the boundary operator $\partial_k$.

\begin{proposition}
Let $Q\subset \R^d$ be a $k$-dimensional elementary cube
$Q = I_1 \times \cdots \times I_d$ and let $I_{i_1}, \cdots, I_{i_k}$ be the nondegenerate components of $Q$
with $I_{i_j}=[l_j, l_j+1]$ for some $l_j\in\Z$.
For each $1 \le j \le k$, let
\begin{align*}
Q_j^- & \;:=\; I_1 \times\dots\times I_{i_j-1} \times [l_j] \times I_{i_j+1} \times\cdots\times I_d \;, \\
Q_j^+ & \;:=\; I_1 \times\dots\times I_{i_j-1} \times [l_j+1] \times I_{i_j+1} \times\cdots\times I_d \;.
\end{align*}
Then
\begin{equation*}
\partial_k \widehat Q \;=\; \sum_{j=1}^k (-1)^{j-1} \Big ( \widehat Q_j^+ - \widehat Q_j^- \Big ) \;.
\end{equation*}
\end{proposition}

We also sum up basic properties of the boundary operator $\partial_k$,
which enable us to define homology groups for cubical sets.

\begin{proposition}\label{1}
The boundary operator satisfies the following properties:
\begin{enumerate}
\item For any $k\in\Z$, it holds that $\partial_{k-1} \circ \partial_k = 0$.
\item For any cubical set $X$ in $\R^d$ and any $k\in\Z$, it holds that
$\partial_k(C_k^d(X)) \subset C_{k-1}^d(X)$. In particular, the operator
$\partial_k^X$ defined as the restriction of $\partial_k$ onto $C_k^d(X)$
is a map from $C_k^d(X)$ to $C_{k-1}^d(X)$.
\end{enumerate} 
\end{proposition}

For a cubical set $X$ in $\R^d$, the cubical chain complex for $X$
is defined as the sequence:
\begin{align*}
\xymatrix{
\cdots\ar[r] & C^d_{k+1}(X)\ar[r]^{\partial^X_{k+1}} & C^d_k(X)\ar[r]^{\partial^X_k~} & C^d_{k-1}(X)\ar[r] & \cdots
}
\end{align*}
We then define the {\em $k$-th homology} $H_k(X) := Z_k(X)/ B_k(X)$
by the quotient $\Z$-module of $Z_k(X):=\ker \partial^X_k$ and $B_k(X) := \mathrm{im}\,\partial^X_{k+1}$.
Note that from Proposition \ref{1} $B_k(X)$ is a submodule of $Z_k(X)$
and thereby $H_k(X)$ is well-defined.
Note also that, if $X$ is bounded, then the homology group $H_k(X)$ is a finitely generated $\Z$-module.
Therefore from the structure theorem for finitely generated $\Z$-modules (see \cite[Corollary 3.1]{kmm}),
the homology group $H_k(X)$ can be represented as $H_k(X) \simeq T_k(X) \oplus \Z^{\b_k(X)}$,
where $T_k(X)$ and $\b_k(X)$ are called the {\em $k$-th torsion} and {\em $k$-th Betti number}, respectively.
Note that, for any cubical set $X$ in $\R^d$ and $k<0$ or $k>d$,
it follows from $C_{k}^d(X)=0$ that $\b_k(X)=0$. Furthermore, $\beta_d(X)$ is also zero for any cubical set
$X\subset\R^d$.  Therefore our attention will be always focused on the case $0\le k < d$.

So far, we have defined the homology $H_k(X)$ for a cubical set $X$ which expresses the $k$-dimensional topological features in $X$. Next, we introduce a generalization of homology defined for an increasing family of cubical sets which characterizes {\em persistent} topological features.

For a bounded cubical set $X$, let $\X=\{X(t)\}_{0\leq t\leq 1}$ be an increasing family of cubical sets $X(t)\subset X$, i.e., $X(s)\subset X(t)$ for any $s\leq t$. It follows from the finiteness that the parameters $t$ at which the cubical sets properly increase $X(t-\epsilon)\subsetneq X(t)$ for any sufficiently small $\epsilon>0$ are finite, and hence we can assign the index $0\leq t_1<\dots <t_n\leq 1$ for those parameters. We recall that the inclusion map $X(s)\hookrightarrow X(t)$ for any $s\leq t$ induces the linear map $\iota^t_s:H_k(X(s))\rightarrow H_k(X(t))$ on homologies by the natural assignment $[z]\mapsto[z]$, where $[z]$ is taken for each equivalent class. Then,
the {\em $k$-th persistent homology} $H_k(\X)=(H_k(X(t)),\iota^t_s)$ of $\X$ is defined by the family of homologies $\{H_k(X(t))\colon 0\leq t\leq 1\}$ and the induced linear maps $\iota^t_s$ for all $s\leq t$ \cite{elz,zc}.

When we replace the coefficient $\Z$ of modules with a field $\Bbbk$, 
the persistent homology satisfies an important structure theorem \cite{zc}. Namely, the persistent homology can be uniquely decomposed as a direct sum
\begin{align*}
H_k(\X)\;\simeq\;\bigoplus_{i=1}^mI(b_i,d_i)\;.
\end{align*}
Here, each summand $I(b_i,d_i)=(U_t,f^s_t)$ called interval representation consists of a family of vector spaces
\begin{align*}
U_t\;=\; \left\{\begin{array}{ll}
\Bbbk,& b_i\leq t<d_i \;,\\
0,&{\rm otherwise}\;,
\end{array}\right.
\end{align*}
and the identity map $f^s_t={\rm id}_\Bbbk$ for $b_i\leq t\leq s<d_i$. Intuitively, each interval $I(b_i,d_i)$ represents a persistent topological feature (i.e., $k$-dimensional hole) which appears and disappears at $t=b_i$ (birth) and $t=d_i$ (death) in $\X$, respectively. Note that the birth $b_i$ and death $d_i$ are given by the indices in $\{t_1,\dots,t_n\}$. The lifetime of the $i$-th interval is defined by $\ell_i=d_i-b_i$, and the lifetime sum of the persistent homology $H_k(\X)$ is given by $L_k(\X)=\sum_{i=1}^m \ell_i$. It is also known \cite{hs1} that the lifetime sum has an alternative expression 
\begin{align*}
L_k(\X)\;=\;\int_0^1\beta_k(X(t))dt\;.
\end{align*}
In this paper, we use the latter expression of the lifetime sum. As we remarked in Section \ref{sec:intro}, the persistent homology and its lifetime sum provide several interesting problems in random topology which 
can be regarded as higher dimensional generalizations of classical subjects in probability theory. We refer to the papers \cite{hs1,hs2} for these problems.

\subsection{Random cubical set}\label{rcs}
We introduce in this subsection our model which describes a wide class of random cubical sets.
We continue to use notation introduced in Subsection \ref{ch}.

Let $\Om$ be the product space $[0,1]^{\K^d}$, called a {\em configuration space},
equipped with the product topology. We denote its Borel $\si$-field by $\F$.
A general element of $\Om$ is denoted by $\om = \{\om_Q\}_{Q \in \K^d}$ and is called a {\em configuration}.
Let $P$ be a stationary and ergodic probability measure on the configuration space $(\Om, \F)$,
that is, $P$ satisfies the following conditions:
\begin{itemize}
\item {\em Probability}: $P$ is a probability measure on $(\Om, \F)$.
\item {\em Stationarity}: $P(\tau_x^{-1}A)=P(A)$ holds for any $x\in\Z^d$ and any $A\in\F$.
\item {\em Ergodicity}: If $\tau_x^{-1}A = A$ for any $x\in\Z^d$ and for some $A\in\F$, $P(A)=0$ or $1$.
\end{itemize}
In the above conditions, $\{\tau_x\}_{x\in\Z^d}$ represents the translation group acting on $\Om$:
\begin{align*}
x + Q & \;:=\; \{x+y : y\in Q\} \;,   & & x \in \Z^d \;, Q \in \K^d \;, \\
\tau_x\om & \;:=\; \{\om_{-x+Q}\}_{Q\in\K^d} \;,   & & \om\in\Om \;, \\
\tau_xA &\;:=\; \{\tau_x\om : \om\in A\} \;, & & A \in \mathcal F \;.
\end{align*}
The expectation with respect to $P$ is denoted by $E[\cdot]$.

We now associate a configuration $\om\in\Om$ with a cubical set  in $\R^d$ as follows.
For an elementary cube $Q\in\K^d$ and time $t\in[0,1]$,
$Q$ is added to a set $X(t) = X_\om(t)$ if $\om_Q \le t$, otherwise, $Q$ is not added.
More precisely, a random cubical set $X(t)$, depending on a configuration $\om$ and $t\in[0,1]$, is defined by 
\begin{equation}\label{rev2}
X(t)  \;=\;  \bigcup_{Q\in\K^d:  \om_Q\le t}Q \;.
\end{equation}
Note that this definition allows some elementary cube $Q\in\mathcal{K}^d$
with $\omega_Q > t$ to be included in the cubical set $X(t)$.
For each $n\in\N$, let $\La_n$ be the rectangle in $\R^d$ given by $\La_n=[-n,n]^d$.
We also set $X^n(t) = X(t) \cap \La_n$.

\begin{remark}\label{rev1}
{\rm The definition \eqref{rev2} perhaps seems to be strange since the association from
a configuration to a cubical set is not one-to-one.
Instead one may consider  the $\Z$-modules spanned by
$\{\widehat Q\in \widehat{\mathcal{K}}^d_k: \omega_Q\le t\}$
for a configuration $\omega\in[0,1]^{\mathcal{K}^d}$.
However such modules are not closed under the boundary operators in general.
One way to avoid this problem is introducing the set of all configurations
$\omega\in[0,1]^{\mathcal{K}^d}$ such that $\om_P\le\om_Q$ for any elementary cubes
$P,Q$ with $P\subset Q$.}
\end{remark}

\begin{remark}{\rm 
To discuss the homology or the Betti number only for fixed $t$,
it might be more suitable to discuss the configuration space $\{0,1\}^{\K^d}$
and the cubical set
\begin{equation*}
X \;=\;  \bigcup_{Q\in\K^d:  \om_Q=1} Q \;,
\end{equation*}
associated with the configuration $\om\in\{0,1\}^{\K^d}$.
One reason to work with $[0,1]^{\K^d}$ is that it provides us with an increasing family $\X=\{X(t)\}_{0\leq t\leq 1}$ of cubical sets, and hence naturally connects to the persistent homology $H_k(\X)$.
}\end{remark}

\subsection{Examples}\label{example}
We introduce in this subsection several examples explained in Subsection \ref{rcs}.
In all examples, we can easily check that the probability measure considered in each example
is a stationary and ergodic probability measure on the configuration space
(see \cite{mr} for the ergodicity).
We refer to the paper \cite{ww} for the study of intrinsic volumes of several models
including parts of Examples \ref{bernoulli} and \ref{closedface}.

\begin{example}\label{bernoulli}
{\rm
Fix integers $d\in\N$ and $0\le k\le d$. 
One simple example of stationary and ergodic measures on $\Om$
is given as the product measure with marginal distributions
\begin{align*}
\begin{cases}
P(\om_Q=0)\;=\; 1 \;, \quad & \text{if }Q\in\K^d_l, l<k \;, \\
P(\om_Q\le t) \;=\; t \;, \quad & \text{if }Q\in\K^d_k, t\in[0,1] \;,  \\
P(\om_Q=1)\;=\; 1 \;, \quad & \text{if }Q\in\K^d_l,  k<l\;.
\end{cases}
\end{align*}
This model is called the {\em Bernoulli random cubical sets} \cite{hs2}.
In this model, for each configuration $\om\in\Om$ and each elementary cube $Q\in\K_k^d$,
$\om_Q$ can be regarded as a birth time of $Q$ in a random cubical set.
We also remark that, for $k=1$ and $0<t<1$, the resultant random cubical set $X(t)$ is known as a
{\em bond percolation model} with parameter $t$.
We refer to the standard text \cite{g} for the mathematical theory of percolation models.
See Figures \ref{fig:bernoulli_2} and \ref{fig:bernoulli_3} for this model with each parameter.
}\end{example}

\begin{figure}[htbp]
\begin{center}
	\includegraphics[width=40mm]{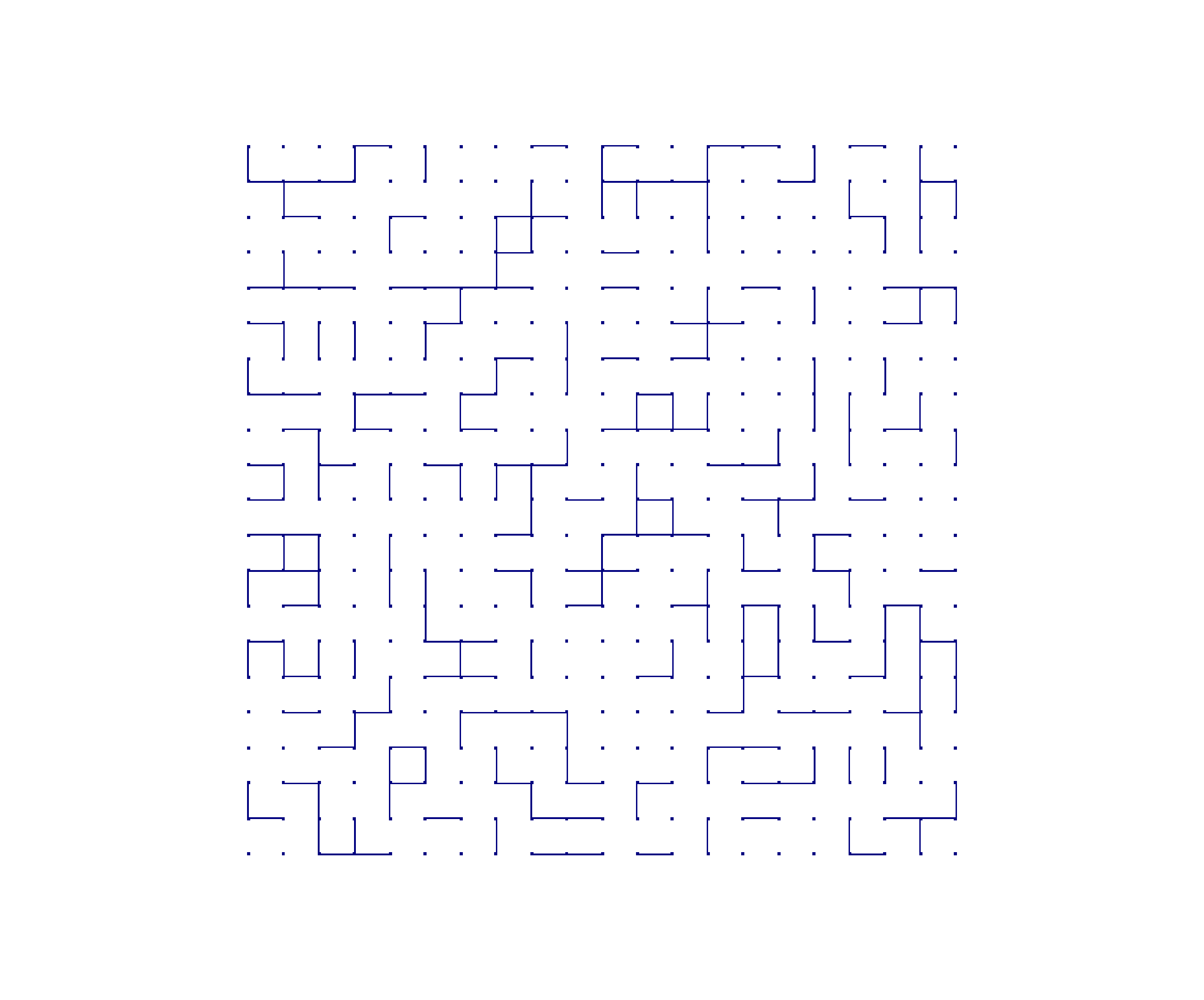}
	\includegraphics[width=40mm]{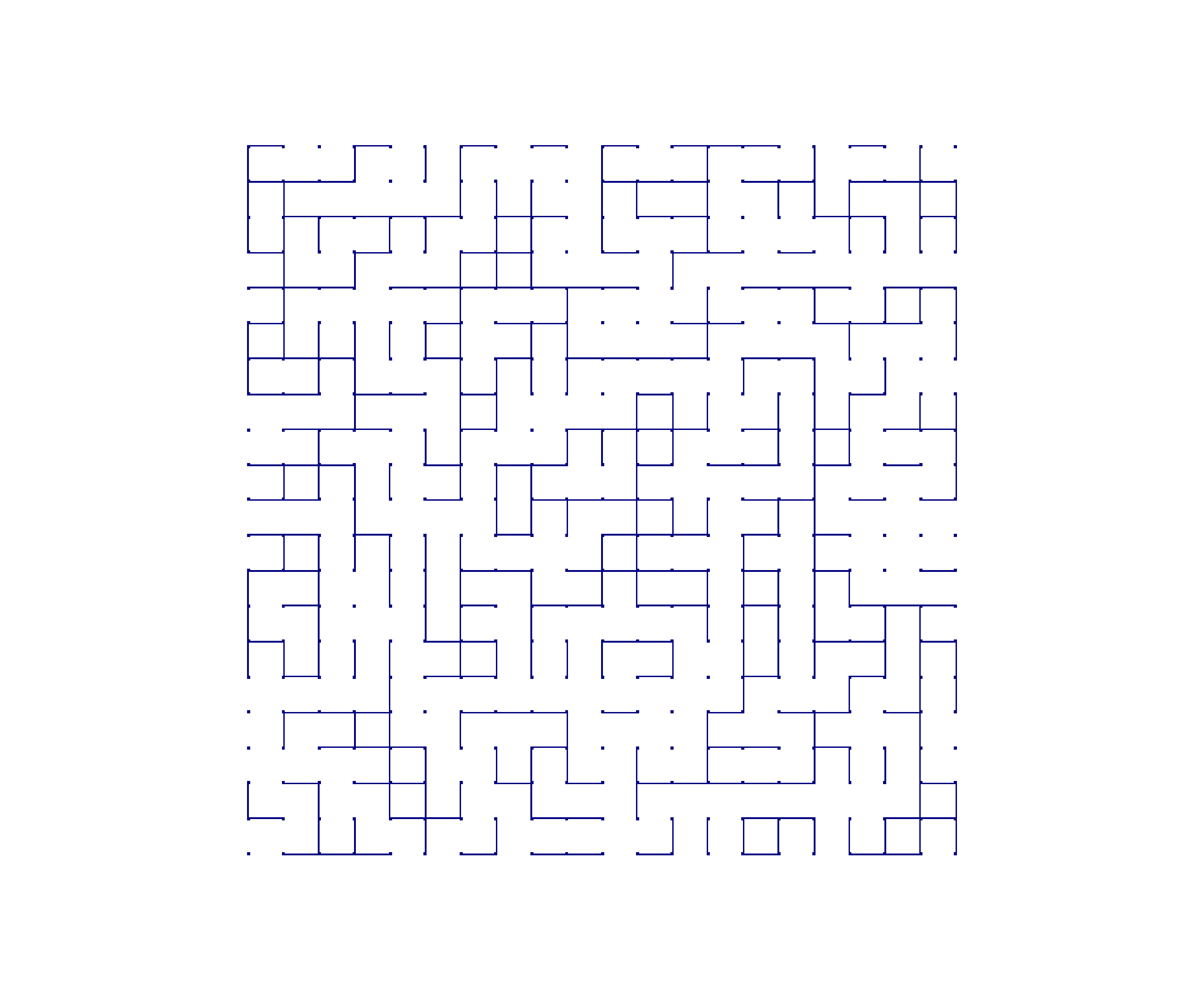}
	\includegraphics[width=40mm]{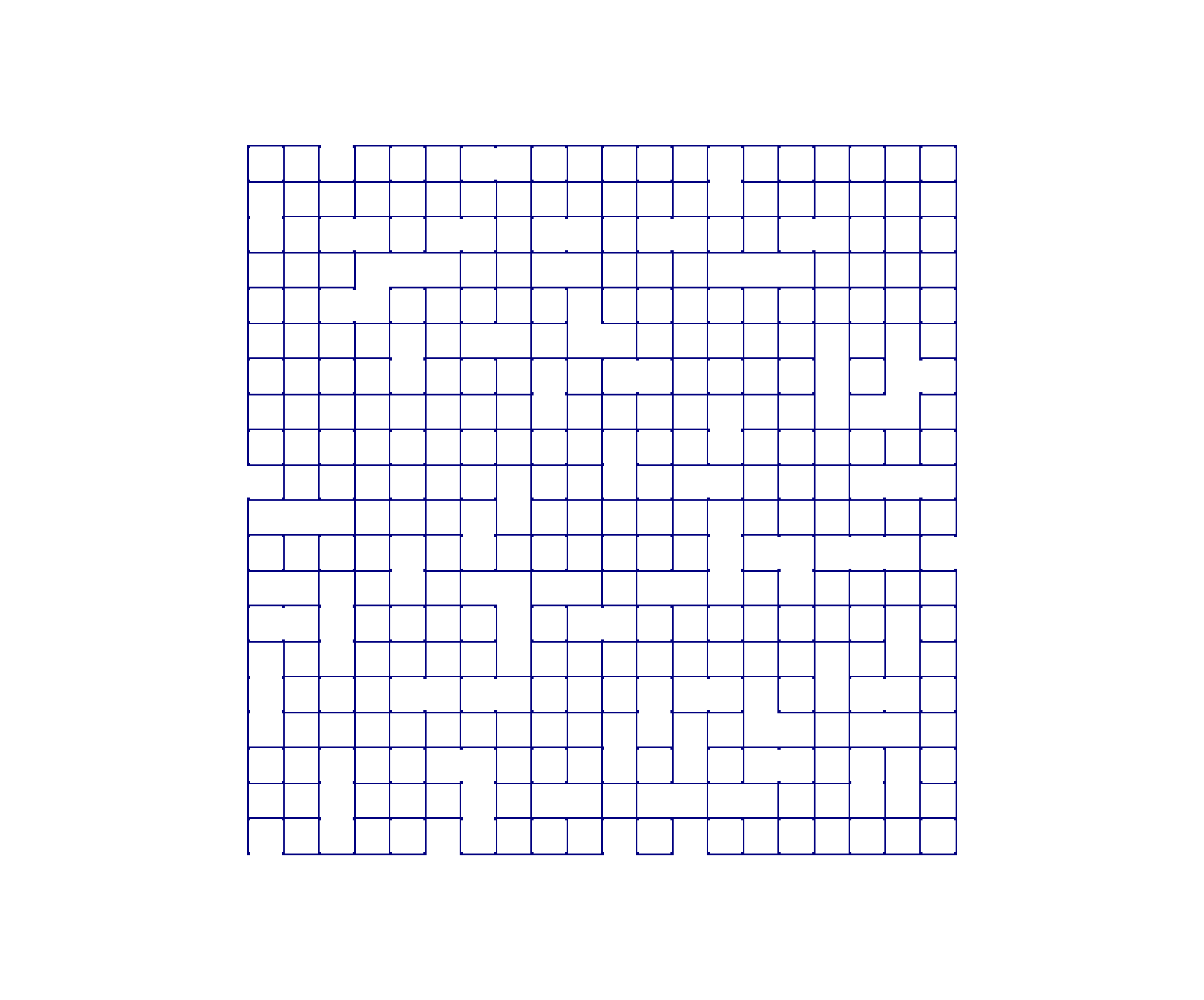}
\end{center}
\caption{$d=2$, $k=1$, $t=0.3$ (left), $0.5$ (middle), $0.9$ (right).}
\label{fig:bernoulli_2}
\end{figure}
\begin{figure}[htbp]
\begin{center}
	\includegraphics[width=40mm]{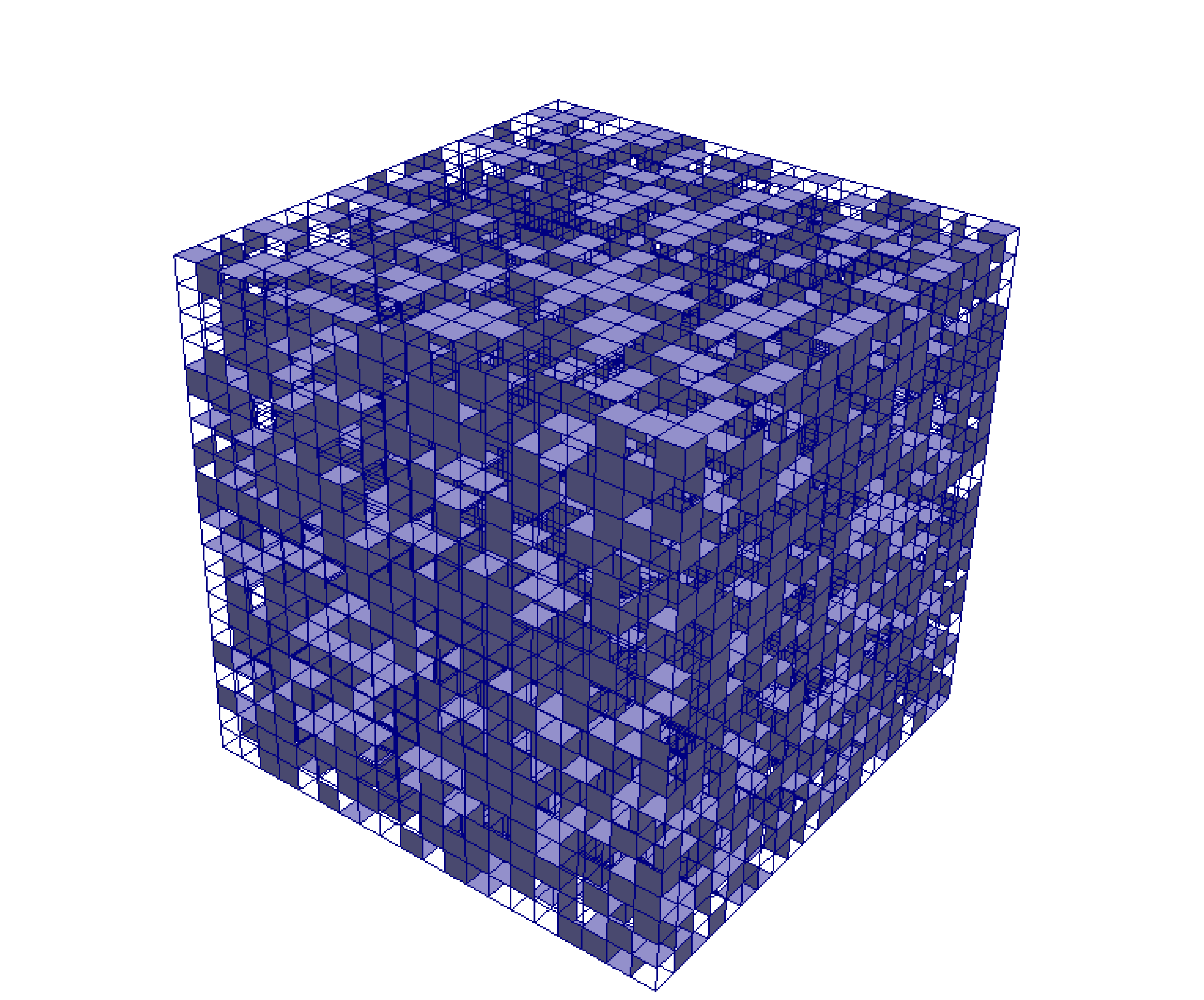}
	\includegraphics[width=40mm]{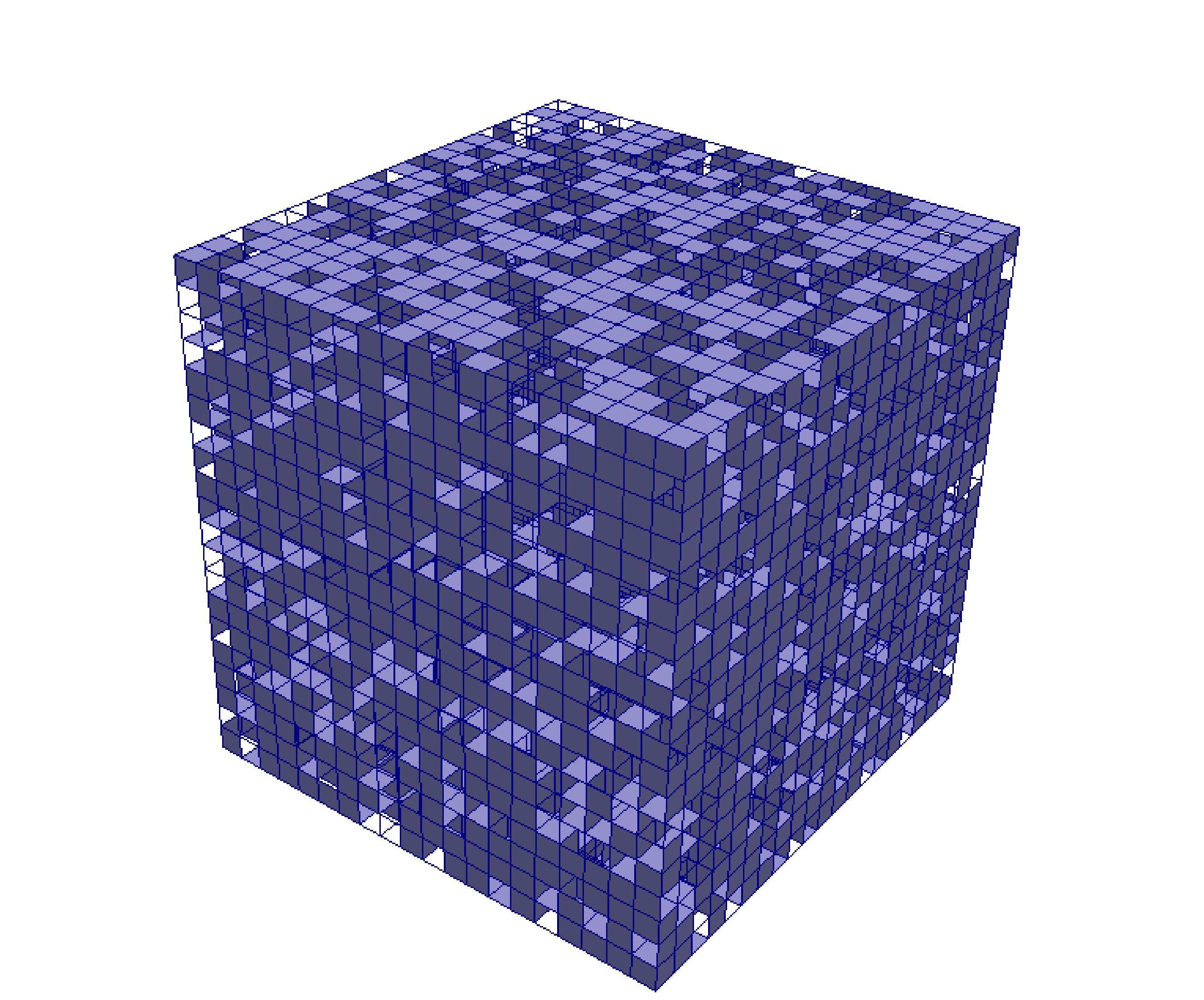}
	\includegraphics[width=40mm]{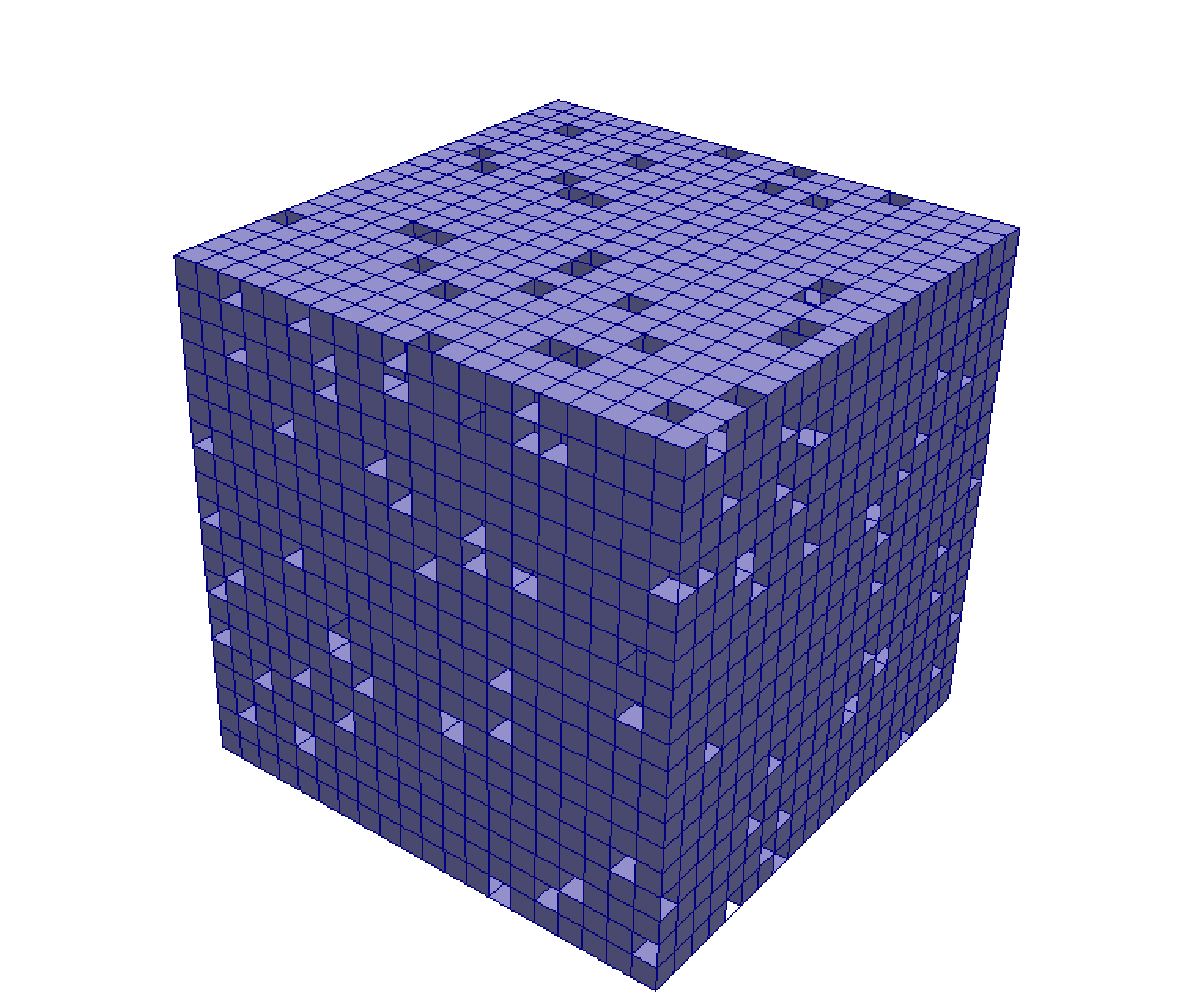}
\end{center}
\caption{$d=3$, $k=2$, $t=0.3$ (left), $0.5$ (middle), $0.9$ (right).}
\label{fig:bernoulli_3}
\end{figure}

\begin{example}\label{closedface}
{\rm 
In the previous example, the resultant cubical set is a union of the $(k-1)$-dimensional complete skeleton and $k$-dimensional elementary cubes.
Another choice of probability measures allows for randomness to occur in all dimensions.
For instance, the product measure $P$ on $\Om$ with marginal distributions
$P(\om_Q\le t)=t$, for any $Q\in\K^d$ and any $t\in[0,1]$, gives such an example.
Note that the probability of the event $\{Q\subset X(t)\}$ is not equal to $t$
if the dimension of $Q$ is less than $d$ (for instance, if $d=1$, then
$P(\{0\}\subset X(t))=1-(1-t)^3\neq t$ in general).
See Figures \ref{fig:tsunoda_2} and \ref{fig:tsunoda_3} for this model with each parameter.
}
\end{example}

\begin{figure}[htbp]
\begin{center}
	\includegraphics[width=40mm]{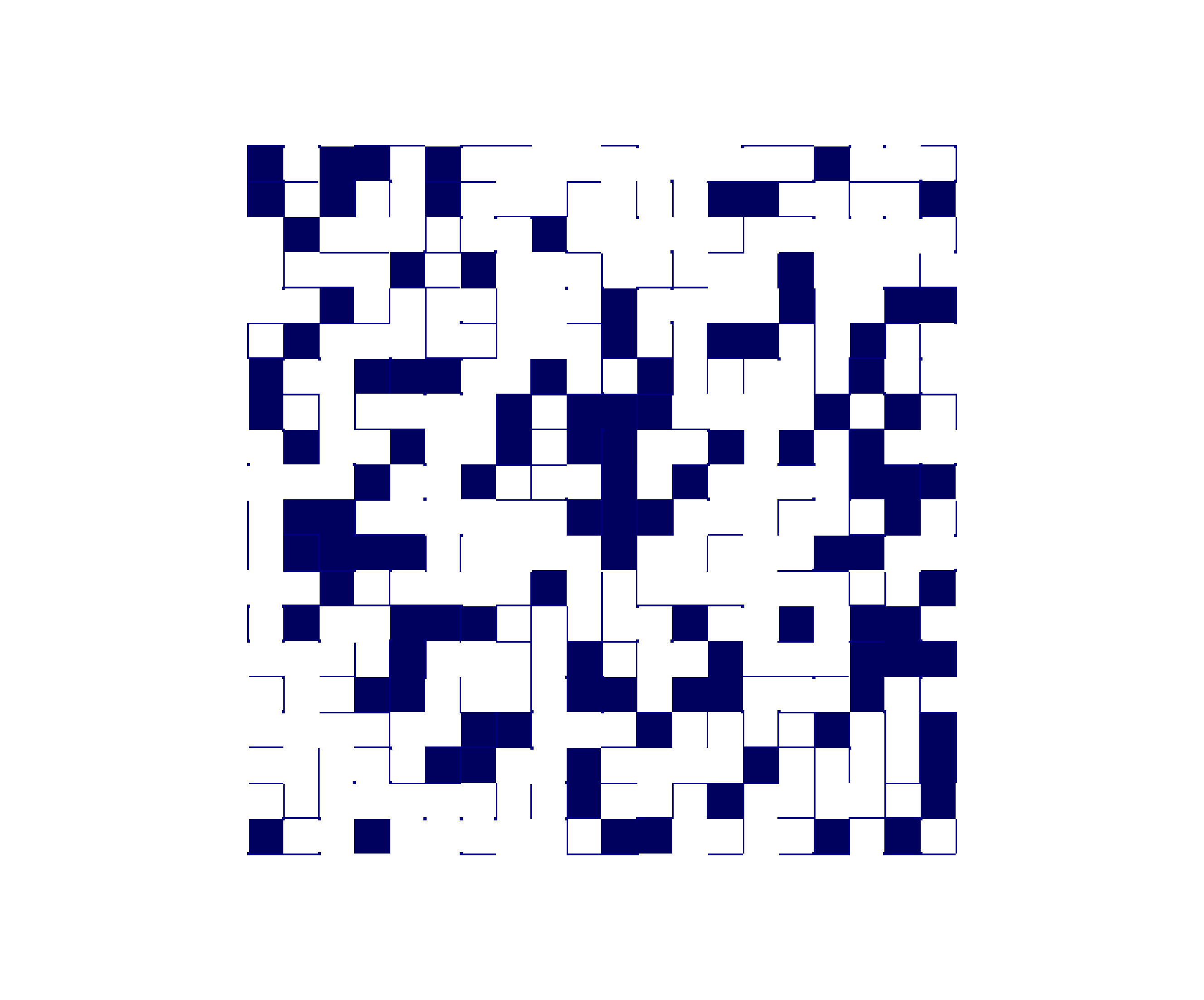}
	\includegraphics[width=40mm]{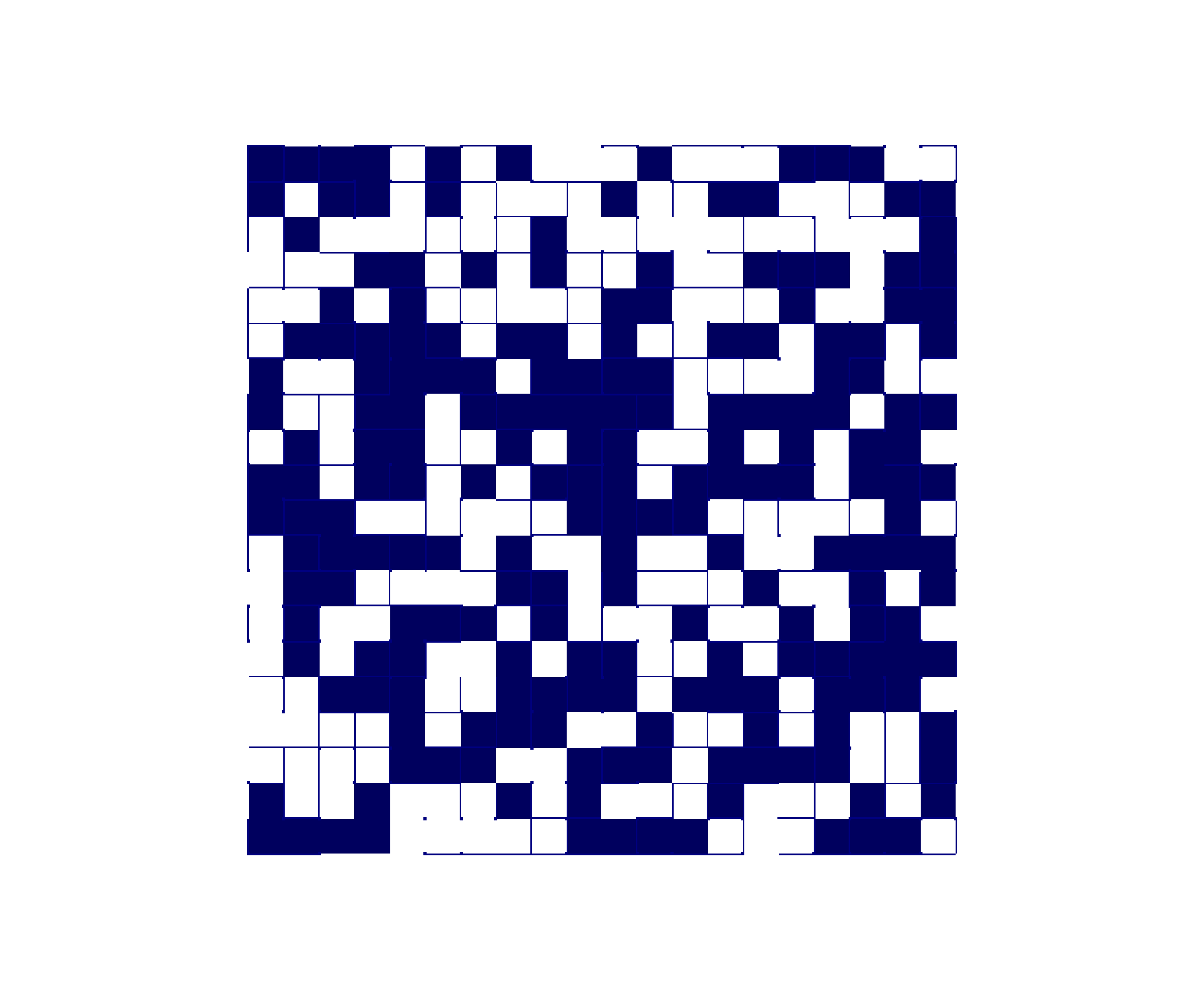}
	\includegraphics[width=40mm]{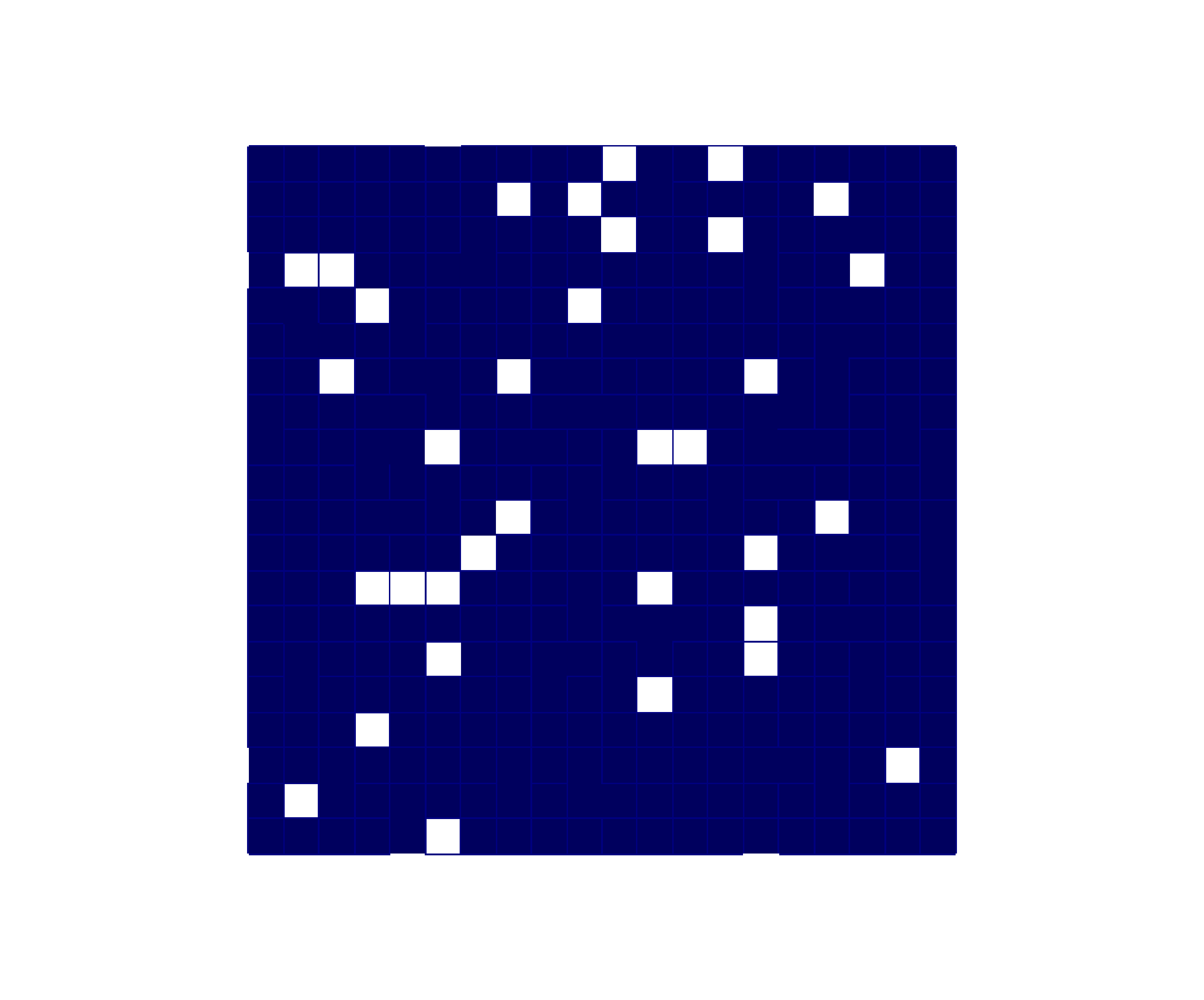}
\end{center}
\caption{$d=2$, $t=0.3$ (left), $0.5$ (middle), $0.9$ (right).}
\label{fig:tsunoda_2}
\end{figure}
\begin{figure}[htbp]
\begin{center}
	\includegraphics[width=40mm]{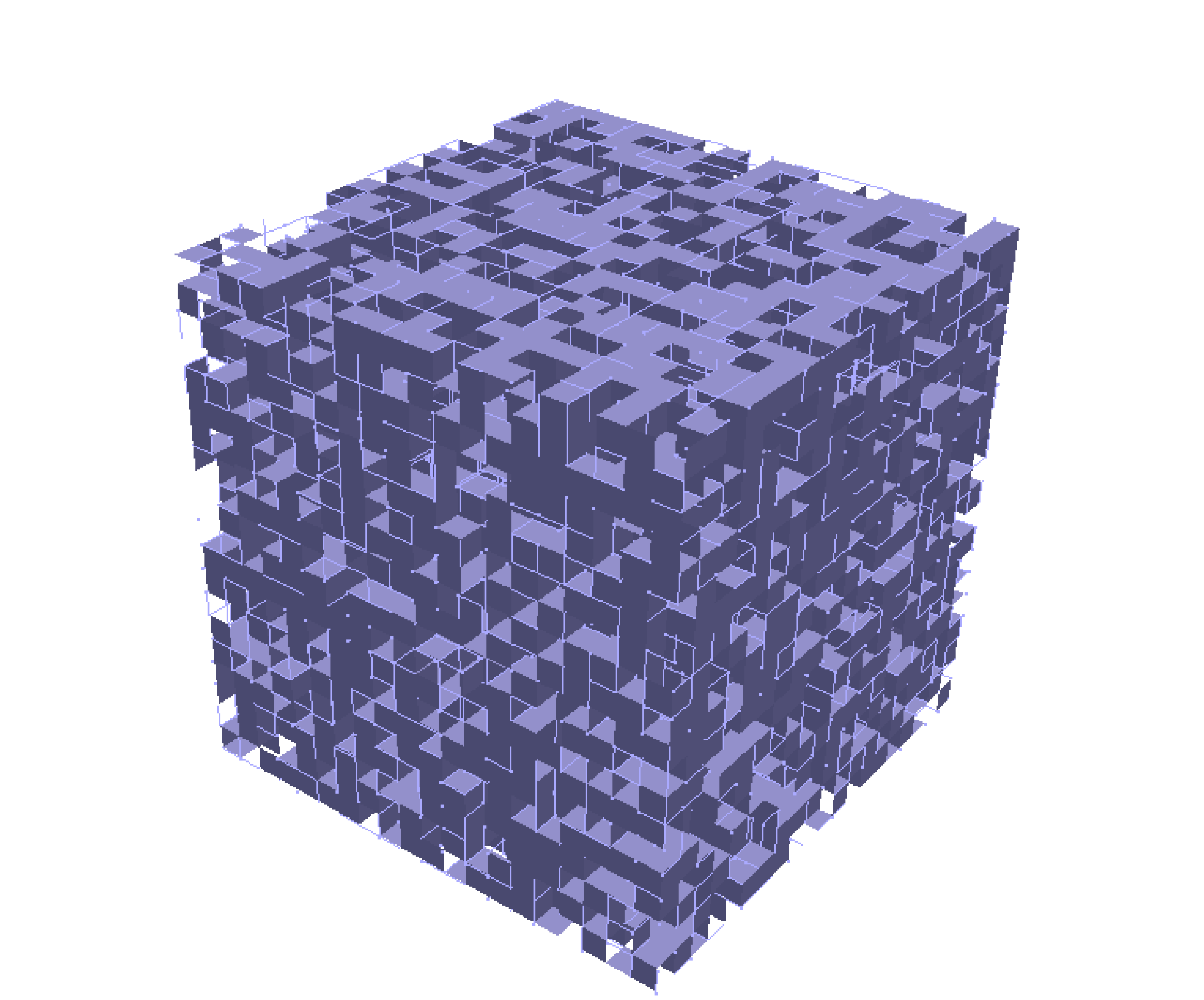}
	\includegraphics[width=40mm]{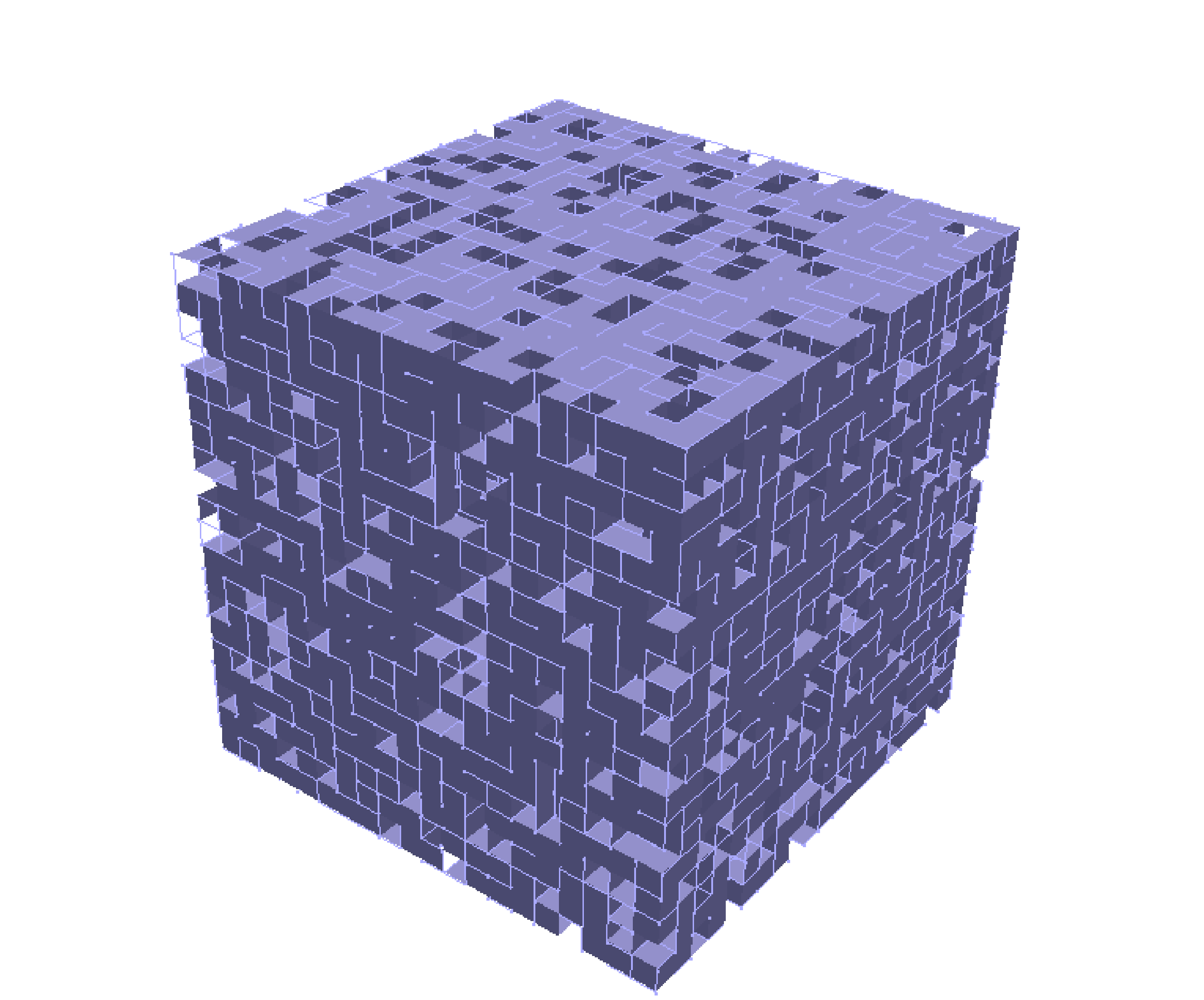}
	\includegraphics[width=40mm]{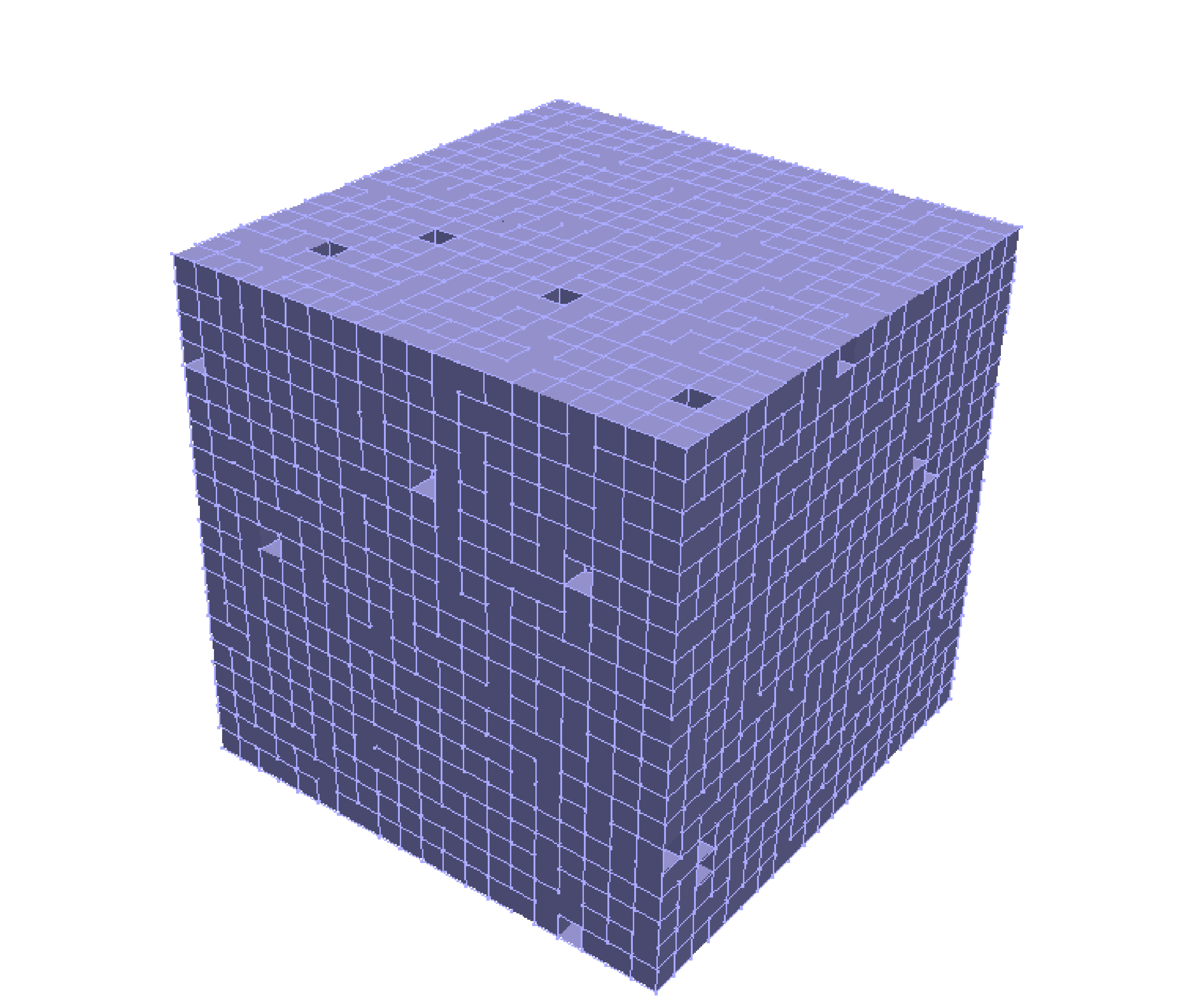}
\end{center}
\caption{$d=3$, $t=0.3$ (left), $0.5$ (middle), $0.9$ (right).}
\label{fig:tsunoda_3}
\end{figure}

\begin{example}\label{cfmodel}
{\rm
We examine in this example another model of random cubical sets,
which is analogous to the random simplicial complex models studied by Costa and Farber \cite{cf1,cf2}.
Fix $d\in\N$ and $0\le p_0, p_1, \cdots, p_d\le 1$.
In this example, we consider the configuration space $\{0,1\}^{\K^d}$ instead of $[0,1]^{\K^d}$.
For a cubical set $X$ in $\R^d$, we define the configuration $\om^X$ by
$\om^X_Q=1$ if $Q\subset X$, otherwise $\om^X_Q=0$.
We now consider the probability function $P_n:\Om\to[0,1]$ given by the formula
\begin{align*}
P_n(\om)\;=\;
\begin{cases}
\displaystyle \prod_{k=0}^d  p_k^{n_k(X)} (1-p_k)^{m_k(X)}&\;, \text{ if $\om=\om^X$ for some cubical set $X$}\;, \\
0 &\;, \text{ otherwise}\;.
\end{cases}
\end{align*}
In the above formula, $n_k(X)$ and $m_k(X)$ stand for
\begin{align*}
n_k(X)\;&:=\; \#\{Q\in\K_k^d: Q\subset X\cap \La_n\}\;, \\
m_k(X)\;&:=\; \#\{Q\in\K_k^d: Q\not\subset X\cap \La_n, \partial_k Q\subset X\cap \La_n\}\;,
\end{align*}
where $\partial_k Q$ is the $k$-th boundary of $Q$:
\begin{equation*}
\partial_k Q=\bigcup_{\tilde Q\in\K_{k-1}^d:\tilde Q\subset Q} \tilde{Q} \;.
\end{equation*}
From the Kolmogorov's extension theorem, there exists
a unique probability measure on the configuration space $\{0,1\}^{\K^d}$,
whose restriction to $\La_n$ coincides with $P_n$ for any $n\in\N$.
We note without proof that one can obtain results analogous to Theorem \ref{lln}
for the Betti numbers defined from this random cubical model.
}\end{example}

\subsection{Main results}
As mentioned in Section \ref{sec:intro}, main concern of this paper is the asymptotic behavior of the Betti numbers or the lifetime sum.
Recall the definition of the random cubical set $X^n(t)$. As it will be clarified in Section \ref{sec:LLN},
the $q$-th Betti number of $X^n(t)$ almost surely increases with the order $|\Lambda_n|=(2n)^d$,
which is the volume of the region $\Lambda_n$ where cubical sets are considered.
This is a consequence of nearly additive property discussed in Lemmas \ref{2}, \ref{3}
and the stationarity of $X^n(t)$.
To see the fluctuation of the $q$-th Betti number of $X^n(t)$ around its mean, 
the necessary normalization is of order $|\Lambda_n|^{1/2}$.
This order comes from the use of general CLT obtained in \cite{p}.
The interested reader may refer to \cite{p} for further details.
The same normalizations are necessary to obtain the LLN and the CLT for the lifetime sum.

To state our results, we introduce some notation.
Recall from Subsection \ref{ch} the definition of the $q$-th Betti number $\b_q(X)$
for a bounded cubical set $X$. Under the probability space $(\Om, \F, P)$,
$\b_q^n(t):=\b_q(X^n(t))$ can be regarded as a real-valued random variable.
We sometimes regard the $q$-th Betti number $\b_q^n(\cdot)$ as a functional
defined for all configurations $\om\in\Om$.
For a Borel subset $A$ in $\R^d$, let $|A|$ denote the Lebesgue measure of $A$.
Our first result states the LLN for the sequence of random variables $\{\b_q^n(t):n\in\N\}$.

\begin{theorem}\label{lln}
Fix integers $0 \le q < d$ and $t\in[0,1]$.
Then there exists a non-random constant $\widehat\b_q(t)$, which depends on $d$, $q$ and $t$, such that
the sequence of random variables $\{|\La_n|^{-1}\b_q^n(t):n\in\N\}$ converges to
$\widehat\b_q(t)$ as $n\to\infty$ almost surely.
\end{theorem}

Recall the notion of the lifetime sum from the last paragraph of Subsection \ref{ch}.
We here define the lifetime sum $L_q^n$ by 
\begin{equation*}
L_q^n  \;=\;  \int_0^1 \b_q^n(t) dt \;.
\end{equation*}
If we ignore the effect of exceptional sets due to the almost sure convergence,
which depend on the continuum parameter $t$,
from Theorem \ref{lln} and the dominated convergence theorem,
one can obtain the convergence result for the lifetime sum $\{|\La_n|^{-1}L_q^n:n\in\N\}$.
Therefore we naturally expect that
\begin{equation*}
\lim_{n\to\infty} \dfrac{1}{|\La_n|}L_q^n \;=\; \int_0^1 \widehat{\b}_q(t) dt \;,
\end{equation*}
almost surely. To make this convergence rigorous, we shall prove the uniform convergence of Betti numbers.

\begin{theorem}\label{ulln}
Fix integers $0 \le q < d$ and $t\in[0,1]$.
Let $\widehat\b_q(t)$ be the almost sure convergent limit that appeared in Theorem \ref{lln}.
Assume that the marginal distribution function
$F^Q(t)=P(\om_Q \le t)$ is continuous in $t\in[0,1]$ for any $Q\in\K^d$.
Then
\begin{equation}\label{10}
\lim_{n\to\infty}\sup_{t\in[0,1]} | \dfrac{1}{|\La_n|} \b_q^n(t) - \widehat\b_q (t) | \;=\; 0 \,,
\end{equation}
almost surely.
\end{theorem}

The LLN for the lifetime sum immediately follows from Theorem \ref{ulln}.

\begin{corollary}\label{14}
Under the assumptions of Theorem \ref{ulln}, it holds that
\begin{equation*}
\lim_{n\to\infty} \dfrac{1}{|\La_n|}L_q^n \;=\; \int_0^1 \widehat{\b}_q(t) dt \;,
\end{equation*}
almost surely.
\end{corollary}

\begin{proof}
The convergence of the sequence $\{|\La_n|^{-1}L_q^n(t):n\in\N\}$ is a direct
consequence of Theorem \ref{ulln}. 
Indeed,
it follows from the definition of $L_q^n$ that
\begin{align*}
|\dfrac{1}{|\La_n|}L_q^n - \int_0^1 \widehat{\b}_q(t) dt|
& \;\le\;
\int_0^1  | \dfrac{1}{|\La_n|} \b_q^n(t) - \widehat\b_q (t) | dt \\
& \;\le\; \sup_{t\in[0,1]} | \dfrac{1}{|\La_n|} \b_q^n(t) - \widehat\b_q (t) | \;.
\end{align*}
From Theorem \ref{ulln}, the last expression vanishes as $n\to\infty$ almost surely,
which completes the proof of Corollary \ref{14}.
\end{proof}

\begin{remark}{\rm 
Hiraoka and Shirai \cite{hs2} showed that the expectation of the lifetime sum
$L_q^n$ is of order $\Theta(|\La_n|)=\Theta(n^d)$
as $n\to\infty$ for cubical complexes shown in Example \ref{bernoulli}.
From Corollary \ref{14} together with the positivity of the limit (cf. Proposition \ref{21}),
we can obtain a refinement of Theorem 3.3 in \cite{hs2}.
}\end{remark}

In the rest of this subsection, we also present the CLT for
the Betti number $\b_q^n(t)$ and the lifetime sum $L_q^n$. To state our results,
we need an additional assumption that the probability measure $P$ is a product measure on $\Om$.
Note that Examples \ref{bernoulli} and \ref{closedface} satisfy this assumption.

\begin{theorem}\label{clt}
Fix integers $0 \le q < d$ and $t\in[0,1]$.
Assume that the probability measure $P$ is a product measure on $\Om$.
Then there exists a constant $\si^2\ge0$, depending on $d,q$ and $t$, such that as $n\to\infty$
\begin{equation*}
\dfrac{1}{|\La_n|} E\Big[\big(\b_q^n(t) - E[\b_q^n(t)]\big)^2\Big] \;\to\; \si^2 \;,
\end{equation*}
and
\begin{equation*}
\dfrac{1}{|\La_n|^{1/2}}\Big (\b_q^n(t) - E[\b_q^n(t)] \Big) \;\Rightarrow\; \mathcal N(0,\si^2) \;,
\end{equation*}
where $\Rightarrow$ denotes convergence in law
and $\mathcal N(0,\si^2)$ stands for the Gaussian distribution with mean $0$ and variance $\si^2$.
\end{theorem}

The proof of Theorem \ref{clt} relies on the general result developed by Penrose \cite{p}.
The idea to prove Theorem \ref{clt} also permits us to obtain the CLT
for the lifetime sum. Recall the definition of the lifetime sum $L_q^n$ introduced after Theorem \ref{lln}.

\begin{theorem}\label{cltforlts}
Fix integers $0 \le q < d$.
Assume that the probability measure $P$ is a product measure on $\Om$.
There exists a constant $\tau^2\ge0$, depending on $d$ and $q$,
such that as $n\to\infty$
\begin{equation*}
\dfrac{1}{|\La_n|} E\Big[\big( L_q^n - E[L_q^n]\big)^2\Big] \;\to\; \tau^2 \;,
\end{equation*}
and
\begin{equation*}
\frac{1}{|\La_n|^{1/2}}\Big( L_q^n - E[L_q^n] \Big) \;\Rightarrow\; \mathcal N(0,\tau^2) \;.
\end{equation*}
\end{theorem}

\subsection{Computations}
In this subsection, we show numerical experiments on asymptotic behaviors
of normalized Betti numbers $|\Lambda_n|^{-1}\beta^n_q(t)$ in order to give intuitive understanding of the main results. 

We use Example \ref{bernoulli} with $d=3$ and $k=2$, and Example \ref{closedface} with $d=3$.  
In both models, we constructed random cubical filtrations 
for $n=10,20,30,40,50,60,70,80$ with 5 samples for each $n$, and computed those Betti numbers. Then, we observed the convergence of the Betti numbers around $n=50\sim 80$. Figure \ref{fig:betti_curves} shows the 1st Betti numbers of one sample only for $n=10, 50, 80$ (top: Example \ref{bernoulli}, bottom: Example \ref{closedface}), and we actually see the overlap for two curves of $n=50$ and $n=80$.
We also observed that the standard deviation of the samples is quite small for large $n$. For example, in the case of Example \ref{bernoulli} with $n=80$, the maximum of the standard deviation for $t\in[0,1]$ is approximately $4.595\times 10^{-4}$.

\begin{figure}[htbp]
\begin{center}
       \includegraphics[width=80mm]{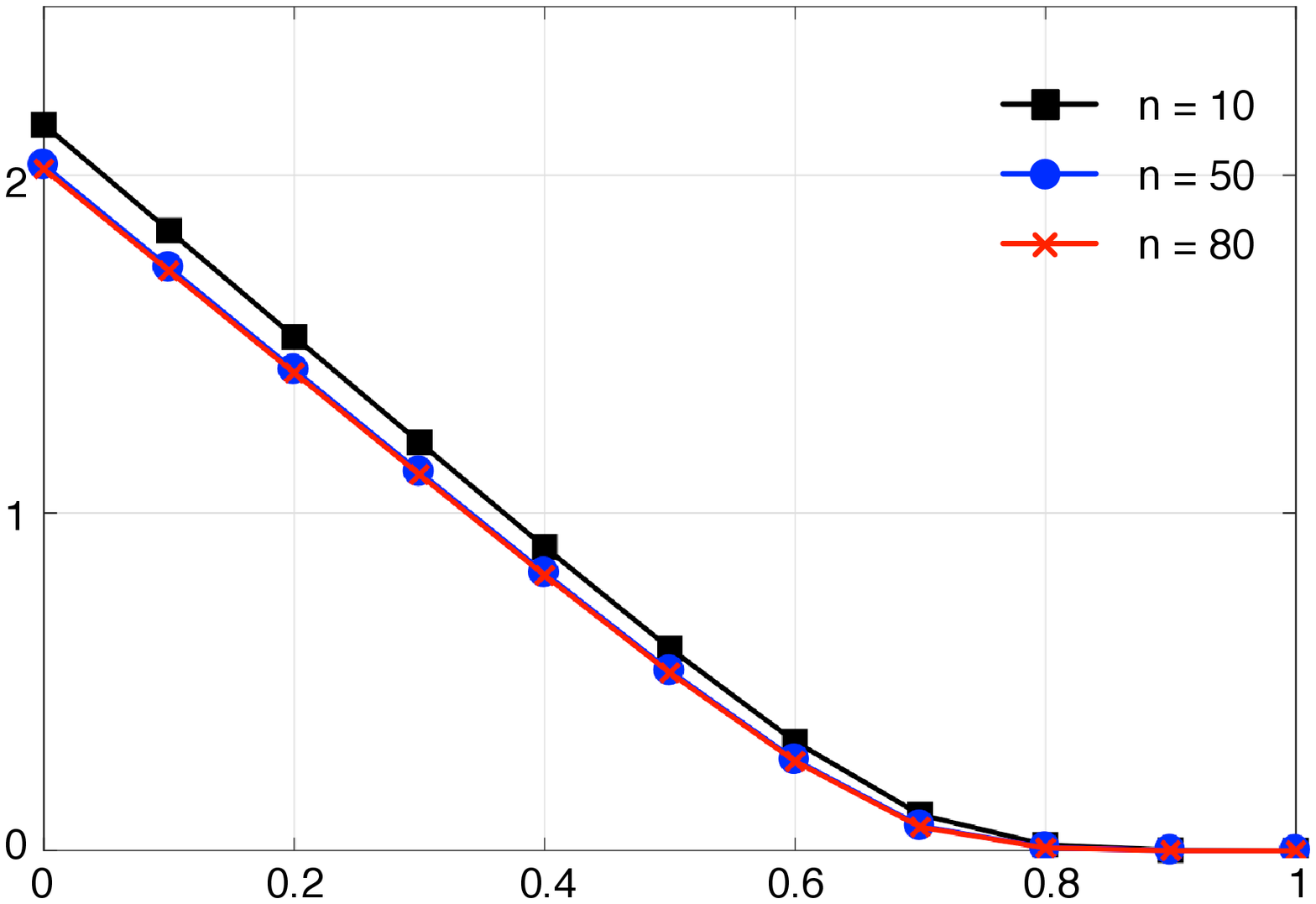}
       \includegraphics[width=80mm]{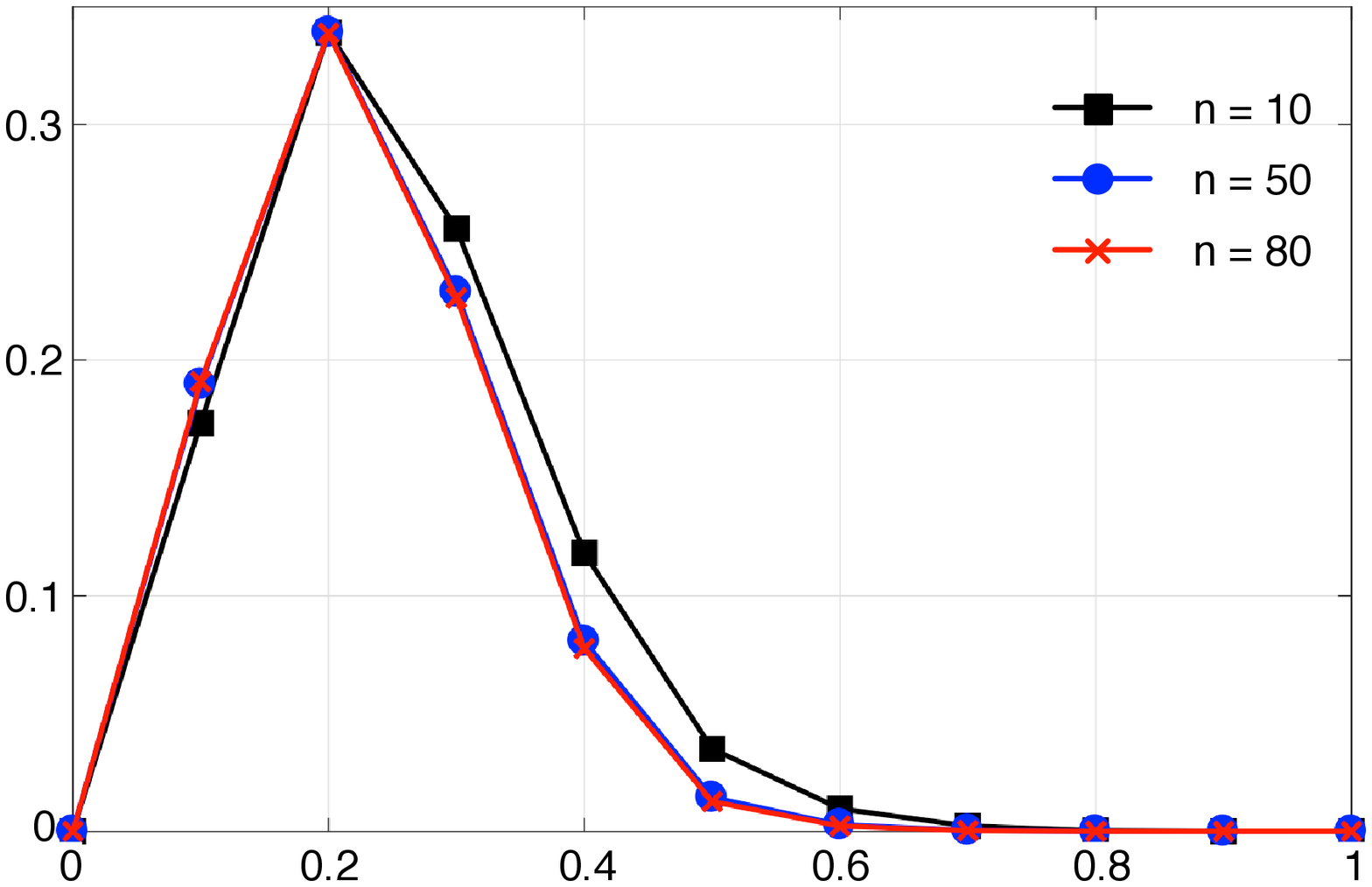}       
\end{center}
    \caption{Computations of normalized Betti numbers 
    $|\Lambda_n|^{-1}\beta^n_q(t)$ for Example \ref{bernoulli} ($d=3$ and $k=2$)
    and Example \ref{closedface} ($d=3$).}
    \label{fig:betti_curves}
\end{figure}

\section{Law of large numbers}\label{sec:LLN}
We prove in this section Theorems \ref{lln} and \ref{ulln}.
Throughout this section, fix integers $0 \le q < d$, where
$d$ corresponds to the dimension of the state space $\R^d$ and
$q$ the dimension of the Betti number $\b_q(\cdot)$, respectively.
For a cubical set $X$ in $\R^d$, let $\#X$ be the number of elementary cubes contained in $X$.

We start with an estimate on the difference of Betti numbers of two cubical sets.
The following lemmas will be repeatedly used in this paper.
For the case of simplicial complexes, see Lemma 2.2 in \cite{ysa}. The extension to the persistent Betti numbers is also shown in Lemma 2.11 in \cite{dhs}.

\begin{lemma}\label{2}
Let $X$ and $Y$ be cubical sets in $\R^d$ with $X \subset Y$. Then
\begin{equation*}
|\b_q(Y) - \b_q(X)| \;\le\; \# Y - \# X \;.
\end{equation*}
\end{lemma}

\begin{proof}
Without loss of generality, we may assume that $\#Y$ is finite.
Assume first that $\# Y - \# X = 1$.
From the definition of the $q$-th Betti number $\b_q(\cdot)$, we have
\begin{equation}\label{15}
\b_q(Y) - \b_q(X) \;=\; \Big (\rank Z_q(Y) - \rank Z_q (X) \Big ) - \Big ( \rank B_q(Y) - \rank B_q(X) \Big) \;.
\end{equation}
Since $\# Y - \# X = 1$, both differences in each of braces on the right hand side of \eqref{15}
are equal to $0$ or $1$. Therefore the conclusion of the lemma holds if $\# Y - \# X = 1$.

For general cubical sets $X$ and $Y$ with $X \subset Y$,
let $X=\cup_{i=1}^m Q_i$ and $Y=\cup_{i=1}^l Q_i$,
where $m=\# X$, $l=\# Y$ and $Q_i$ is an elementary cube in $\R^d$ for each $1 \le i \le l$.
Then it follows from the first part of the proof that
\begin{align*}
|\b_q(Y) - \b_q(X)|
& \;\le\; \sum_{i=m+1}^l
|\b_q(X\cup Q_{m}\cup\cdots\cup Q_i) - \b_q(X\cup Q_{m}\cup\cdots\cup Q_{i-1})| \\
& \;\le\; \sum_{i=m+1}^l  1 \;=\; \# Y - \# X \;,
\end{align*}
which completes the proof of Lemma \ref{2}.
\end{proof}

\begin{lemma}\label{3}
Let $X$ and $Y$ be cubical sets with $X\subset Y$.
Assume that there exists a cubical set $Z=\cup_{i=1}^m Q_i$
such that $Y\setminus X\subset Z$ and $\dim Q_i = d$ for any $i=1,\cdots, m$.
Then
\begin{equation*}
\#Y - \#X \;\le\; 3^d|Z| \;.
\end{equation*}
\end{lemma}
\begin{proof}
Note that $\#Y-\#X$ can be written as
\begin{align*}
\#Y-\#X
&\;=\; \sum_{Q\in\K^d: Q\subset Y\cap Z} 1 + \sum_{Q\in\K^d: Q\subset Y, Q\not\subset Y\cap Z} 1  \\
&\;-\; \sum_{Q\in\K^d: Q\subset X\cap Z} 1 -  \sum_{Q\in\K^d: Q\subset X, Q\not\subset X\cap Z} 1 \;.
\end{align*}
From the assumption $Y\setminus X\subset Z$, any elementary cube $Q$ with
$Q\subset Y$ and $Q\not\subset Y\cap Z$ must be a subset of $X$.
Therefore second and fourth sums cancel each other.
Hence we have
\begin{align}\label{23}
\#Y-\#X \;\le\; \sum_{Q\in\K^d: Q\subset Z} 1\;.
\end{align}
Let $\binom dk$ be the binomial coefficient.
Since the number of $k$-dimensional elementary cubes in a unit cube $[0,1]^d$
is equal to $\binom dk 2^{d-k}$ for any $0\le k\le d$, the number of elementary cubes in $[0,1]^d$ is equal to $3^d$.
This fact together with \eqref{23} completes the proof of Lemma \ref{3}.
\end{proof}

We now turn to the proof of Theorem \ref{lln}.

\begin{proof}[Proof of Theorem \ref{lln}]
Let $\widehat\b_q(t)$ be the limit supremum of the sequence
$\{ |\La_n|^{-1}E[\b_q^n(t)]:n\in\N\}$:
\begin{equation}\label{19}
\widehat\b_q(t) \;:=\; \varlimsup_{n\to\infty} \dfrac{1}{|\La_n|} E[\b_q^n(t)] \;.
\end{equation}
It follows from Lemmas \ref{2} and \ref{3} that
\begin{align}\label{22}
\dfrac{1}{|\La_n|}\b_q^n(t) \;\le\; \dfrac{1}{|\La_n|}\# X^n(t)
\;\le\; \dfrac{3^d}{|\La_n|}|\La_n| \;=\; 3^d \;.
\end{align}
Therefore $\widehat\b_q(t)$ is finite.
We hereafter show that
\begin{equation}\label{4}
\lim_{n\to\infty} \dfrac{1}{|\La_n|} \b_q^n(t) \;=\; \widehat\b_q(t) \;,
\end{equation}
almost surely.

Fix a positive integer $K$ and take an integer $m$, which depends on $K$ and $n$, 
satisfying the inequalities
\begin{equation}\label{5}
(K+1)m \;\le\; n \;<\; (K+1)(m+1) \;.
\end{equation}
From \eqref{5}, one can easily see that
\begin{equation}\label{6}
|\La_K|m^d \;\le\; |\La_n| \;\le\; |\La_K|m^d + 6dm|\La_n|^{1-1/d} \;.
\end{equation}

Let $\mathbb{I}_m:=\{1, \cdots, m \}^d$.
For a multiindex $i=(i_1,\cdots, i_d)$ in $\mathbb{I}_m$,
let $z^i=(z^i_1,\cdots, z_d^i)$ be the point in $\Z^d$ with
$z^i_j=-n+(K+1)(2i_j-1)$ and let $\La_K^i=z^i+\La_K$.
Define the cubical set $X^{m,K}(t)$ by
\begin{equation*}
X^{m,K}(t) \;=\; \bigsqcup_{i\in \mathbb{I}_m} \{Q\in\K^d: \om_Q \le t, Q\subset \La_K^i \} \;, 
\end{equation*}
and denote by $\b_q^{m,K}(t)$ the $q$-th Betti number of $X^{m,K}(t)$.
From the triangle inequality, $| |\La_n|^{-1}\b_q^n(t) - \widehat\b_q(t) |$
is bounded above by
\begin{equation}\label{7}
\begin{split}
| \dfrac{1}{|\La_n|}\b_q^n(t) - \dfrac{1}{|\La_K|m^d}\b_q^n(t) | 
& \;+\;  | \dfrac{1}{|\La_K|m^d}\b_q^n(t) - \dfrac{1}{|\La_K|m^d}\b_q^{m,K}(t) | \\
& \;+\;  | \dfrac{1}{|\La_K|m^d}\b_q^{m,K}(t) - \widehat\b_q(t) | \;.
\end{split}
\end{equation}

In view of \eqref{22} and \eqref{6}, the first term of \eqref{7} vanishes as $n\to\infty$ and $K\to\infty$.
On the one hand, from Lemmas \ref{2} and \ref{3}, the second term of  \eqref{7} can be bounded above by
\begin{equation}\label{8}
\dfrac{1}{|\La_K|m^d}\{\# X^n(t) - \# X^{m,K}(t) \}
\;\le\; \dfrac{3^d}{|\La_K|m^d}|\La_n \setminus (\bigsqcup_{i\in \mathbb{I}_m} \La_K^i)| \;.
\end{equation}
From the latter inequality of \eqref{6}, the right hand side in \eqref{8} vanishes as $n\to\infty$ and $K\to\infty$.
Since subsets $\{\La_K^i:i\in \mathbb I_m\}$  are mutually disjoint, 
$\b_q^{m,K}(t)$ can be written as
\begin{equation*}
\b_q^{m,K}(t) \;=\; \sum_{i\in \mathbb I_m} \b_q(\La_K^i) \;.
\end{equation*}
Therefore it follows from the multivariate ergodic theorem \cite[Proposition 2.2]{mr} that
\begin{align*}
\lim_{m\to\infty}\dfrac{1}{m^d}\b_q^{m,K}(t) \;=\; E[\b_q^K(t)] \;,
\end{align*}
almost surely. Hence
\begin{align*}
\varlimsup_{K\to\infty}\varlimsup_{n\to\infty}
| \dfrac{1}{|\La_K|m^d}\b_q^{m,K}(t) - \widehat\b_q(t) |
\;=\; \varlimsup_{K\to\infty} | \dfrac{1}{|\La_K|}E[\b_q^K(t)] - \widehat\b_q(t) | \;=\; 0 \;,
\end{align*}
by choosing appropriate subsequence in $K$ if necessary. 

These limits prove \eqref{4}, and therefore the proof of Theorem \ref{lln} is completed.
\end{proof}

Before proceeding the proof of Theorem \ref{ulln},
we prove the uniform convergence of the scaled expectation
of the Betti numbers $\{|\La_n|^{-1}E[\b_q^n(\cdot)]:n\in\N\}$.

\begin{lemma}\label{9}
Assume that the marginal distribution function
$F^Q(t)=P(\om_Q \le t)$ is continuous in $t\in[0,1]$ for any $Q\in\K^d$.
Then the sequence of functions $\{|\La_n|^{-1}E[\b_q^n(\cdot)]:n\in\N\}$
uniformly converges to $\widehat\b_q(\cdot)$.
\end{lemma}

\begin{proof}
From the dominated convergence theorem, Theorem \ref{lln} and \eqref{22},
the sequence of functions $\{|\La_n|^{-1}E[\b_q^n(\cdot)]:n\in\N\}$ pointwisely
converges to $\widehat\b_q(\cdot)$. 
We shall show below that  $\{|\La_n|^{-1}E[\b_q^n(\cdot)]:n\in\N\}$ is an equicontinuous sequence.
Therefore the conclusion of Lemma \ref{9} follows from the Ascoli-Arzel\`a theorem.

For each $0 \le k \le d$, denote by $F_k$ the distribution function $F^Q$
with some $Q\in\K_k^d$. Note that from the stationarity of $P$
this definition does not depend on a choice of $Q\in\K_k^d$.
Let $0 \le s \le t \le 1$. From Lemma \ref{2}, it holds that
\begin{equation*}
\dfrac{1}{|\La_n|}| \b_q^n(t) - \b_q^n(s) |
\;\le\; \dfrac{1}{|\La_n|}\sum_{Q\in\K^d: Q\subset \La_n} {\bf 1}\{ s <  t_Q \le t \} \;.
\end{equation*}
Since $E[{\bf 1}\{ s <  \om_Q \le t \}] = F_k(t)-F_k(s)$ for each $Q\in\K_k^d$
and the number of elementary cubes contained in $\La_n$ is less than or equal to $3^d|\La_n|$,
\begin{align}\label{16}
|\dfrac{1}{|\La_n|} E[\b_q^n(t)] - \dfrac{1}{|\La_n|} E[\b_q^n(s)] |
\;\le\; 3^d \sum_{k=0}^d \Big(F_k(t) - F_k(s)\Big) \;.
\end{align}
Since the right hand side of \eqref{16} does not depend on $n\in\N$
and vanishes as $t-s\downarrow0$ from continuity of the marginal distribution functions,
we obtain the desired equicontinuity.
\end{proof}

\begin{remark}\label{12}
{\rm 
It follows from the proof of Lemma \ref{9} that
$\widehat\b_q$ is uniformly continuous on $[0,1]$. 
}\end{remark}

\begin{proof}[Proof of Theorem \ref{ulln}]
Let $m$ be a positive integer and set
$\lfloor t \rfloor_m=\lfloor tm \rfloor / m$ for $0\le t \le 1$,
where $\lfloor \cdot \rfloor$ stands for the floor function.
Then the inside of the limit on the left hand side of \eqref{10} can be bounded above by
\begin{equation}\label{11}
\begin{split}
\dfrac{1}{|\La_n|}\sup_{0\le t \le 1} | \b_q^n(t) - \b_q^n(\lfloor t \rfloor_m) | 
& \;+\; \sup_{0 \le t \le 1} | \dfrac{1}{|\La_n|}\b_q^n(\lfloor t \rfloor_m) - \widehat\b_q(\tm) | \\
& \;+\; \sup_{0 \le t \le 1} | \widehat\b_q(\tm) - \widehat\b_q(t) | \;.
\end{split}
\end{equation}

The second supremum of \eqref{11} can be rewritten as
\begin{equation*}
\max_{i=0,\cdots, m} | \dfrac{1}{|\La_n|}\b_q^n(\frac{i}{m}) -\widehat\b_q(\frac{i}{m}) | \;.
\end{equation*}
Therefore the last expression almost surely converges to $0$ as $n\to\infty$ for any fixed $m\in\N$.
It also follows from Remark \ref{12} that the third supremum of \eqref{11}  vanishes as $m\to\infty$.
Hence to conclude the proof it is enough to show that the first supremum of \eqref{11}
vanishes as $n\to\infty$ and $m\to\infty$.

From Lemma \ref{2}, the first supremum of \eqref{11} is bounded above by
\begin{equation*}
\sup_{0 \le t \le 1} \dfrac{1}{|\La_n|} \sum_{Q\in\K^d: Q\subset \La_n} {\bf 1}\{t-1/m \le \om_Q \le t \} \;.
\end{equation*}
This last expression is also bounded above by
\begin{equation}\label{13}
\max_{i=0, \cdots, m-1} \dfrac{1}{|\La_n|} \sum_{Q\in\K^d: Q\subset \La_n} {\bf 1}\{i/m \le \om_Q \le (i+2)/m \} \;.
\end{equation}
From the multivariate ergodic theorem \cite[Proposition 2.2]{mr},
the limit supremum in $n\in\N$ of the expression \eqref{13} is almost surely bounded above by
\begin{equation*}
C \max_{i=0, \cdots, m-1} \sum_{k=0}^d 
\Big( F_k(\min{(1,(i+2)/m)}) - F _k(i/m) \Big) \;.
\end{equation*}
for some constant $C>0$, depending only on $d$.
Therefore the first supremum of \eqref{11} vanishes as $n\to\infty$ and $m\to\infty$,
and it completes the proof of Theorem \ref{ulln}.
\end{proof}

In the rest of this section, we give some sufficient condition which ensures positivity
of the limit $\widehat\b_q(t)$ that appeared in Theorem \ref{lln}.
We start from giving notation needed for describing its sufficient condition.

For $x\in\Z^d$ and $K\in\N$, denote by $\mathcal{L}_{x,K}$ the subset of $\K^d$
given by $\{Q\in\K^d: Q\subset x+\La_{K}\}$.
For a finite subset $\mathcal{L}\subset\K^d$ with $\mathcal{L}_{x,K}\subset\mathcal{L}$, define
random cubical sets $X_\L(t)$ and $X_{\L}^{x,K}(t)$ by
\begin{align*}
X_\L(t) &\;:=\; \bigcup\{Q\in\L:\om_Q\le t \} \;, \\
X_\L^{x,K}(t) &\;:=\; \bigcup\{Q\in\L\setminus\L_{x,K}:\om_Q\le t \} \;.
\end{align*}
Let $\Om_q(x,K,t)\subset\Om$ be the set of all configurations satisfying the inequality
\begin{equation}\label{20}
\b_q(X_\L(t)) \;\ge\; 1+\b_q(X_\L^{x,K}(t)) \;,
\end{equation}
for any finite subset $\L\subset\K^d$ with $\L_{x,K}\subset\L$.

The following proposition asserts that $\widehat\b_q(t)$ is positive if
the event $\Om_q(0,K,t)$ occurs for some $K\in\N$ with positive probability.

\begin{proposition}\label{21}
Fix  integers $0\le q<d$ and $t\in[0,1]$. If there exists a positive integer $K$ with $P(\Om_q(0,K,t))>0$,
then $\widehat\b_q(t)>0$.
\end{proposition}

\begin{proof}
Set $n=(K+1)m$ with $m\in\N$. Let $\mathbb{I}_m:=\{1, \cdots, m \}^d$.
For a multiindex $i=(i_1,\cdots, i_d)$ in $\mathbb{I}_m$,
let $z^i$ be the point in $\Z^d$ $z^i=(z^i_1,\cdots, z_d^i)$ with $z^i_j=-n+(K+1)(2i_j-1)$.

Since subsets $\{z^i+\La_K: i\in \mathbb{I}_m\}$ of $\R^d$ are mutually disjoint,
therefore by using \eqref{20} repeatedly we have
\begin{equation*}
\dfrac{1}{|\La_n|}\b_q(X^n(t)) \;\ge\; \dfrac{1}{|\La_n|}\sum_{i\in\mathbb{I}_m} {\bf 1} \{\om\in\Om_q(z^i,K,t)\} \;.
\end{equation*}
From the stationarity of $P$, by taking the expectation we also have
\begin{equation}\label{24}
\dfrac{1}{|\La_n|}E[\b_q(X^n(t))] \;\ge\; \dfrac{1}{(2K+2)^d}P(\Om_q(0,K,t)) \;.
\end{equation}
Therefore, from \eqref{19}, letting $n\to\infty$ in \eqref{24} gives us
\begin{equation*}
\widehat\b_q(t) \;\ge\; \dfrac{1}{(2K+2)^d}P(\Om_q(0,K,t)) \;>\;  0 \;,
\end{equation*}
and it completes the proof of Proposition \ref{21}.
\end{proof}

\section{Central limit theorem} \label{sec:CLT}

We prove in this section Theorems \ref{clt} and \ref{cltforlts}.
The proof of the CLT for Betti numbers relies on
a general method developed by Penrose \cite{p}.
We first state it in our situation for the sake of completeness.

We first re-parametrize configurations by $x\in\Z^d$ instead of $\K^d$
as follows. Let $\mathcal{N}^d$ be the set of all elementary cubes in $\R^d$
whose left most point is equal to the origin $O$, that is,
\begin{equation*}
\mathcal{N}^d \;=\; \{ N \in\K^d: \min_{a\in I_i(N)} a = 0 \} \;.
\end{equation*}
Note that the cardinality of $\mathcal{N}^d$ is equal to $2^d$.
We identify  $\{\om_Q: Q\in \K^d\}$ with $\{\om_x=(\om_{x,N} :N\in\mathcal{N}^d): x\in\Z^d\}$,
through the unique decomposition $Q=x+N$, $x\in\Z^d$ and $N\in\mathcal{N}^d$.
Thus the random cubical set $X(t)$ can be defined from a random element $\{\om_x:x\in\Z^d\}$.

Let $(\om^*_{0,N}:N\in\mathcal N_q^d)$ be an independent copy of $\om_0$
and set
\begin{align*}
\om^*_x \;=\;
\begin{cases}
(\om^*_{0,N}:N\in\mathcal N^d) \;, \qquad & \text{if $x=0$ } \;, \\
(\om_{x,N}:N\in\mathcal N^d) \;, \qquad & \text{otherwise } \;.
\end{cases}
\end{align*}
We denote by $X^*(t)$ and $X^{*,n}(t)$ the random cubical sets obtained from
random variables $\{\om^*_x=(\om^*_{x,N} :N\in\mathcal{N}^d): x\in\Z^d\}$
in a similar manner as we defined for $X(t)$ and $X^n(t)$.

Let $\B$ be the collection of all subsets $B$ of $\Z^d$
such that $B=(x+\La_n)\cap \Z^d$ for some point $x\in\Z^d$ and $n\in\N$.
Denote by $\mathcal F_O$ the $\si$-field generated by
$\{ \om_x : x \preceq O\}$, where $ x \preceq y$ means $x$ precedes or equals
$y$ in the lexicographic ordering on $\Z^d$.
For a family of real-valued random variables $(H(\om;B), B\in\B)$,
define $(D_OH)(B)$, which is the \lq\lq effect of changing $\om_0$\rq\rq,  as
\begin{equation*}
(D_OH)(B)\;=\;H(\om;B) - H(\om^*;B) \;.
\end{equation*}

The following result is obtained by Penrose \cite{p}.

\begin{theorem}\label{pclt}
Let $(H(\om;B), B\in\B)$ be a family of real-valued random variables indexed by $\B$,
satisfying the following conditions:
\begin{itemize}
\item Translation invariance:  $H(\tau_x\om;x+B)=H(\om;B)$ for all $x\in\Z^d$, $\om\in\Om$ and
$B\in\B$.
\item Stability: There exists a random variable $D_H(\infty)$
such that for any sequence $\{A_n:n\in\N\}$ in $\B$ with $\varliminf A_n =\Z^d$,
random variables $\{ (D_OH)(A_n):n\in\N\}$ converge in probability to $D_H(\infty)$ as $n\to\infty$.
\item Bounded moment condition: There exists a constant $\ga>2$ such that
\begin{equation*}
\sup_{B\in\B} E\Big[\big|(D_OH)(B)\big|^\ga\Big] \;<\; \infty \;.
\end{equation*}
\end{itemize}
Then, for any sequence $\{A_n:n\in\N\}$ in $\B$ with $\varliminf A_n =\Z^d$, as $n\to\infty$
\begin{equation*}
\dfrac{1}{|\La_n|} E\Big[\Big( H(\om;A_n) -  E[H(\om;A_n)]\Big)^2\Big] \;\to\; \si^2 \;,
\end{equation*}
and
\begin{equation*}
\dfrac{1}{|\La_n|^{1/2}} \Big(H(\om;A_n) - E[H(\om;A_n)]\Big) \Rightarrow \mathcal N(0,\si^2) \;,
\end{equation*}
with $\si^2=E[(E[D_H(\infty)|\mathcal{F}_O])^2]$.
\end{theorem}

We  turn to proving the stabilization property for the Betti number.
The following proposition asserts that the Betti number for any cubical sets stabilizes
in the deterministic setting.

\begin{proposition}\label{sta}
Let $X$ and $Y$ be cubical sets in $\R^d$ such that the symmetric difference $X\triangle Y$ forms a bounded set.
Then there exists a constant $\De_\infty\in\Z$ such that,
for any sequence $\{A_n:n\in\N\}$ in $\B$ with $\varliminf A_n =\Z^d$,
there exists $n_\infty\in\N$ such that
\begin{equation*}
\b_q(X \cap A_n) - \b_q(Y \cap A_n) \;=\;  \De_\infty \;,
\end{equation*}
for any $n\ge n_\infty$.
\end{proposition}
\begin{proof}
We first claim that there exist constants $\De_\infty\in\Z$ and $n'_\infty\in\N$
such that, for any $n\ge n'_\infty$,
\begin{equation}\label{17}
\b_q(X \cap \La_n) - \b_q(Y \cap \La_n) \;=\;  \De_\infty \;.
\end{equation}
Let $W$ be the cubical set $X\cap Y$.
From the definition of the $q$-th Betti numbers, the left hand side of \eqref{17} can be written as
\begin{equation}\label{18}
\begin{split}
\Big(\rank Z_q(X_n) -\rank Z_q(W_n)\Big) &- \Big(\rank Z_q(Y_n) -\rank Z_q(W_n)\Big) \\
-  \Big(\rank B_q(X_n) -\rank B_q(W_n)\Big) &+ \Big(\rank B_q(Y_n) -\rank B_q(W_n)\Big) \;,
\end{split}
\end{equation}
where $X_n:=X\cap \La_n$, $Y_n:=Y\cap \La_n$  and $W_n:=W\cap \La_n$.
Hence, it suffices to show the stability for each of braces in \eqref{18}. 
Here we only prove for $\rank Z_q(X_n) -\rank Z_q(W_n)$ and $\rank B_q(X_n) -\rank B_q(W_n)$  because of the symmetry.

From the definition of $W$ and the assumption on $X\triangle Y$,
it is clear that $W_n\subset X_n$ and
$\#X_n - \#W_n$ is bounded in $n$.
On the one hand, by the same way as in the proof of  Lemma \ref{2}, we have
\begin{equation*}
\rank Z_q(X_n) -\rank Z_q(W_n) \;\le\; \#X_n - \#W_n \;.
\end{equation*}
Therefore $\rank Z_q(X_n) - \rank Z_q(W_n)$ is bounded in $n$.

Fix integers $n, m$ with $n\le m$. Let us consider the map
\begin{align*}
f:\frac{Z_q(X_n)}{Z_q(W_n)} \;\longrightarrow\; \frac{Z_q(X_m)}{Z_q(W_m)} \;,
\qquad [c]\longmapsto [c] \;. 
\end{align*}
Since
\begin{equation*}
\begin{array}{ccc}
X_n & \subset & X_m \\
\rotatebox{90}{$\subset$} & & \rotatebox{90}{$\subset$}    \\
W_n & \subset & W_m
\end{array}
\end{equation*}
$f$ is well-defined. If $c$ belongs to $Z_q(X_n)\cap Z_q(W_m)$, then
$c$ can be written as
\begin{equation*}
c\;=\; \sum_{\widehat Q\in\widehat\K_q^d(X_n)}\lan c,\widehat Q\ran\widehat Q
\;=\; \sum_{\widehat Q\in\widehat\K_q^d(W_m)}\lan c,\widehat Q\ran\widehat Q \;.
\end{equation*}
Since $\{\widehat Q:\widehat Q\in \widehat\K_q^d(X_m)\}$ is a basis of $C^d_q(X_m)$
and $W_n$ is a subset of $X_m$, $\lan c,\widehat Q \ran$ vanishes if $\widehat Q\notin\widehat\K_q^d(W_n)$.
Therefore $c$ belongs to $Z_q(W_n)$.
Hence $f$ is injective and thereby $\rank Z_q(X_n) - \rank Z_q(W_n)$ is nondecreasing in $n$.

For the stability on $\rank B_q(X_n)-\rank B_q(W_n)$, we study the map
\begin{align*}
g:\frac{B_q(X_n)}{B_q(W_n)} \;\longrightarrow\; \frac{B_q(X_m)}{B_q(W_m)} \;,
\qquad [c]\longmapsto [c] \;
\end{align*}
for $n\leq m$. We claim that this map is surjective for sufficiently large $n$ and $m$ with $n\le m$. Let $\partial_{q+1}c\in B_q(X_m)$ with $c\in C_{q+1}(X_m)$. Then, by taking sufficiently large $n$ and $m$ with $n\le m$, $d$ can be expressed as $c=c_1+c_2$ with $c_1\in C_{q+1}(W_m)$ and $c_2\in C_{q+1}(X_n)$, since $X\triangle Y$ is bounded. This implies  the surjectivity of the map $g$. This claim leads to the nonincreasing property of $\rank B_q(X_n)-\rank B_q(W_n)$, and hence shows its stability.

We now deal with the general case.
Let $\{A_n:n\in\N\}$ be a sequence in $\B$ with $\varliminf A_n =\Z^d$.
Take $n_\infty\ge n'_\infty$ such that $\La_{n'_\infty}\subset A_n$
for any $n\ge n_\infty$. We now claim that
\begin{equation*}
\b_q(X \cap A_n) - \b_q(Y \cap A_n) \;=\;  \De_\infty \;,
\end{equation*}
for any $n\ge n_\infty$.
By the same reason as explained in the first part of the proof,
to conclude the claim it is enough to show that
\begin{equation}\label{last}
\rank Z_q(X \cap A_n) -\rank Z_q(W \cap A_n) \;=\; \rank Z_q(X_{n'_\infty}) -\rank Z_q(W_{n'_\infty}) \;,
\end{equation}
for any $n\ge n_\infty$.
Take $l\in\N$, depending on $n$, such that $A_n\subset \La_{n'_\infty + l}$.
Let us consider the injections
\begin{align}\label{last2}
\frac{Z_q(X_{n'_\infty})}{Z_q(W_{n'_\infty})} \;\hookrightarrow\; \frac{Z_q(X\cap A_n)}{Z_q(W\cap A_n)}
\;\hookrightarrow\; \frac{Z_q(X_{n'_\infty+l})}{Z_q(W_{n'_\infty+l})} \;.
\end{align}
From the first part of the proof, the rank of the lower set of \eqref{last2} coincides
with the one of the upper set,
and this concludes the proof of Proposition \ref{sta}.
The proofs for the other cases are similar.
\end{proof}

\begin{proof}[Proof of Theorem \ref{clt}]
To conclude the proof of Theorem \ref{clt},
from Theorem \ref{pclt}, it is enough to show that
the functional $\b_q(\om,t)$ satisfies three conditions stated in Theorem \ref{pclt}.

The translation invariance obviously follows from the definition of $\b_q(\cdot)$
and the stabilization property also follows from Proposition \ref{sta}.
Let $X^{0,n}(t)$ be the cubical set $X^n(t)\cap X^{*,n}(t)$.
Note  that $X^{0,n}(t)\subset X^n(t)$ and $X^{0,n}(t)\subset X^{*,n}(t)$.
Hence it follows from Lemma \ref{2} that
\begin{align}\label{26}
|D_O\b_q(\La_n, t)| & \;\le\; | \b_q(X^n(t)) - \b_q(X^{0,n}(t)) | 
+ | \b_q(X^{0,n}(t)) - \b_q(X^{*,n}(t)) | \notag\\
& \;\le\; 2\#\mathcal{N}^d \;=\; 2^{d+1} \;,
\end{align}
Therefore the bounded moment condition is shown, and it completes the proof of Theorem \ref{clt}.
\end{proof}

\begin{proof}[Proof of Theorem \ref{cltforlts}]
The proof of Theorem \ref{cltforlts} is similar to the one of Theorem \ref{clt}.
We need to prove three conditions stated in Theorem \ref{pclt} for the lifetime sum
as a family of real-valued random variables indexed by $\B$ .
We only check the stabilization property because other conditions are easy to check,
and are left to the readers.

Denote the lifetime sum $L_q^n$ by $L_q(\om, \La_n)$.
Then from the definition of $L_q^n$
\begin{align}\label{25}
D_OL_q(\La_n) \;=\; \int_0^1 \b_q(X^n(t)) - \b_q(X^{*,n}(t)) dt\;. 
\end{align}
From \eqref{26}, the integrand on the right hand side of \eqref{25} is bounded in $n$.
Moreover, from Proposition \ref{sta}, the integrand on the right hand side of \eqref{25}
converges as $n\to\infty$. Therefore it follows from the dominated convergence theorem
that $D_OL_q(\La_n)$ converges as $n\to\infty$, which proves the stabilization property
for the lifetime sum.

\end{proof}

\section{Conclusions}\label{sec:conclusions}
In this paper, we have shown the LLN and CLT
for Betti numbers and lifetime sums of random cubical sets in $\R^d$.
Then, a next interesting problem is to show those limiting theorems for persistence diagrams
on random cubical filtrations. Recently, the paper \cite{dhs} developed a random measure theory
which guarantees the limiting persistence diagrams by using limiting persistence Betti numbers.
Hence, to show the limiting persistence diagram,
we need to generalize the results in this paper to persistence Betti numbers. 
Furthermore, in connection with the paper \cite{ww},
it would be an interesting problem to derive the explicit expression of $\widehat\b_q(t)$.

\section*{Acknowledgements}
The authors wish to thank to Trinh Khanh Duy, Ippei Obayashi, and Tomoyuki Shirai
for their valuable suggestions and stimulating comments.


\begin{thebibliography}{99}
\bibitem{akkmop} Arai, Z., Kalies, W., Kokubu, H., Mischaikow, K., Oka, H., Pilarczyk, P.: 
A Database schema for the analysis of global dynamics of  multiparameter systems.
SIAM J. APPL. DYN. SYST. {\bf 8}, 757--789 (2008)

\bibitem{bk} Bobrowski, O., Kahle, M.: Topology of random geometric complexes: a survey. J. Appl. and Comput. Topology (2018). https://doi.org/10.1007/s41468-017-0010-0

\bibitem{cf1} Costa, A., Farber, M.: Large random simplicial complexes, I.
Preprint https://arxiv.org/abs/1503.06285
 
\bibitem{cf2} Costa, A., Farber, M.: Large random simplicial complexes, II;
the fundamental group. Preprint https://arxiv.org/abs/1509.04837
  
\bibitem{dhs} Duy, T.K., Hiraoka, Y., Shirai, T.:
Limit theorems for persistence diagrams. Accepted in Annals of Applied Probability. 
  
\bibitem{eh} Edelsbrunner, H., Harer, H.: Computational topology. An introduction.
Amer. Math. Soc., Providence (2010)

\bibitem{elz} Edelsbrunner, H., Letscher, D., Zomorodian, A.: 
Topological persistence and simplification. Discrete Comput. Geom. {\bf 28}, 511--533 (2002)

\bibitem{f} Frieze, A.M.: On the value of a random minimum spanning tree problem. 
Discrete Appl. Math. {\bf 10}, 47--56 (1985)
	
\bibitem{g} Grimmett, G.: Percolation. Springer-Verlag, Berlin (1999)

\bibitem{hk} Hino, M., Kanazawa, S.: Asymptotic behavior of lifetime sums for random simplicial complex processes. arXiv:1802.00548. 

\bibitem{hs1} Hiraoka, Y., Shirai, T.: Minimum spanning acycle and lifetime of
persistent homology in the Linial-Meshulam process. Preprint https://arxiv.org/abs/1503.05669

\bibitem{hs2} Hiraoka, Y., Shirai, T.: Tutte polynomials and random-cluster models
in Bernoulli cell complexes. Preprint https://arxiv.org/abs/arXiv:1602.04561
  
\bibitem{kmm} Kaczynski, T., Mischaikow, K., Mrozek, M.:
Computational Homology. Springer-Verlag, New York (2004)

\bibitem{k} Kahle, M.: Topology of random simplicial complexes: a survey. 
In: Algebraic topology: applications and new directions.
Contemp. Math. {\bf 620} (Tillmann, U., Galatius, S., Sinha, D. eds.).
pp. 201--221. Amer. Math. Soc., Providence (2014) 

\bibitem{kimura} 
Kimura, M., Obayashi, I., Takeichi, Y., Murao, R., Hiraoka, Y.: Non-empirical identification of trigger sites in heterogeneous processes using persistent homology. Scientific Reports 8, 3553 (2018).
	
\bibitem{kms} Kurtuldu, H., Mischaikow, K., Schatz, M.:
Extensive Scaling from Computational Homology and Karhunen-Lo{\`e}ve
Decomposition Analysis of Rayleigh-B{\'e}nard Convection Experiments.
Phys. Revi. Lett. {\bf 107}, 034503 (2011)

\bibitem{lm} Linial, N., Meshulam, R.:
Homological connectivity of random $2$-complexes.
Combinatorica {\bf 26}, 475--487 (2006)
  
\bibitem{mr} Meester, R., Roy, R.: Continuum Percolation.
Cambridge University Press, Cambridge (1996)
  
\bibitem{p} Penrose, M.D.: A central limit theorem with applications to percolation, 
epidemics and Boolean model. Ann. Probab. {\bf 29}, 1515--1546 (2001)

\bibitem{ww} Werman, M., Wright, M.L.: Intrinsic volumes of random cubical complexes.
Discrete Comput. Geom. {\bf 56}, 93--113 (2016)
  
\bibitem{ysa} Yogeshwaran, D., Subag, E., Adler, R.J.: Random geometric complexes
  in the thermodynamic regime. Probab. Theory Relat. Fields (2015). doi:10.1007/s00440-015-0678-9
  
\bibitem{zc} Zomorodian, A., Carlsson, G.: Computing persistent homology.
Discrete Comput. Geom. {\bf 33}, 249--274 (2005)

\end{thebibliography}
\end{document}